\documentclass[11pt,reqno]{amsart}
\usepackage{amscd,graphicx,psfrag,amsxtra}
\usepackage{pinlabel}
\usepackage{hyperref}

\newcommand{\tw}{\tilde{\xi}}  
\newcommand{\ii}{\mathbf{I}\,}  % index form
\newcommand{\xtilde}{\tilde X}

\newcommand{\p}{\partial}
\newcommand{\ddt}{\frac{\partial}{\partial t}}
\newcommand{\Z}{\mathbb Z}
\newcommand{\cc}{\mathcal C}
\newcommand{\E}{\mathcal E}

\newcommand{\Q}{\mathbb Q}
\renewcommand{\H}{\mathcal H}
\newcommand{\I}{\mathbb I}
\newcommand{\K}{\mathcal K}

\newcommand{\D}{\mathcal D}

\newcommand{\Sm}{\mathcal S}

\newcommand{\C}{\mathbb C}

\newcommand{\R}{\mathbb R}

\newcommand{\rr}{\mathfrak R} %Riemannian curvature appearing in Lichnerowicz formula

\newcommand{\ep}{\varepsilon}
\newcommand{\eq}{\sim}
\renewcommand{\phi}{\varphi}

\newcommand{\tr}{\operatorname{tr}}
\newcommand{\Tr}{\operatorname{Tr}\,}
\newcommand{\btr}{\operatorname{Tr^{\flat}}\,}
\newcommand{\sbtr}{\operatorname{Str^{\flat}}\,}
\newcommand{\str}{\operatorname{str}\,}
\newcommand{\ind}{\operatorname{ind}}
\newcommand{\coker}{\operatorname{coker}}
\renewcommand{\Re}{\operatorname{Re}}

\newcommand{\Hom}{\operatorname{Hom}}
\newcommand{\spin}{\,\operatorname{spin}}
\newcommand{\supp}{\,\operatorname{supp}}

\newcommand{\sign}{\operatorname{sign}}

\newcommand{\txp}{\tilde X_+}

\newcommand{\xp}{X_+}
\newcommand{\xm}{X_-}
\newcommand{\zp}{Z_{\infty}}

\newcommand{\met}{\operatorname{{\mathcal R}}}
\newcommand{\w}{\widehat}
\newcommand{\ch}{\operatorname{ch}\,}

\newcommand{\conn}{\mathbin{\#}} % or {\,{\hash}\,}

\newcommand{\spinc}{\ifmmode{\operatorname{Spin}^c}\else{$\operatorname{spin}^c$\ }\fi}

\newcommand\barf{f} % just f for now
\newcommand{\diff}{\operatorname{Diff}}          %{\mathit{Diff}}
\newcommand\rep{\alpha}   % name for a representation
\newcommand{\mm}{\mathfrak M}
\newcommand{\mmp}{\mm^+}
\newcommand{\rrp}{\rr^+} 
\newcommand{\sss}{{S^2\hskip-2pt\times\hskip-2pt S^2}}

\newtheorem{theorem}{Theorem}[section]

\newtheorem{lemma}[theorem]{Lemma}
\newtheorem{proposition}[theorem]{Proposition}
\newtheorem{corollary}[theorem]{Corollary}
\newtheorem{thm}{Theorem}

\theoremstyle{definition}

\newtheorem{remark}[theorem]{Remark}

\title{An index theorem for end-periodic operators}
\thanks{The first author was partially supported by NSF Grant 0805841, the second author was partially supported by NSF Grant 1105234, and the third author was partially supported by NSF Grant 1065905 and the Max-Planck-Institut f\"ur Mathematik in Bonn, Germany}
\author[Tomasz Mrowka]{Tomasz Mrowka}
\address{Department of Mathematics\newline\indent Massachusetts Institute of 
Technology, Cambridge MA 02139}
\email{\rm{mrowka@mit.edu}}
\author[Daniel Ruberman]{Daniel Ruberman}
\address{Department of Mathematics, MS 050\newline\indent Brandeis
University \newline\indent Waltham, MA 02454}
\email{\rm{ruberman@brandeis.edu}}
\author[Nikolai Saveliev]{Nikolai Saveliev}
\address{Department of Mathematics\newline\indent
University of Miami \newline\indent PO Box 249085
\newline\indent Coral Gables, FL 33124}
\email{\rm{saveliev@math.miami.edu}}
\begin{document}
\begin{abstract}
We extend the Atiyah, Patodi, and Singer index theorem for first order differential operators from the context of manifolds with cylindrical ends to manifolds with periodic ends. This theorem provides a natural complement to Taubes' Fredholm theory for general end-periodic operators. Our index theorem
is expressed in terms of a new periodic eta-invariant that equals the Atiyah-Patodi-Singer eta-invariant in the cylindrical setting.  We apply this periodic eta-invariant to the study of moduli spaces of Riemannian metrics of positive scalar curvature.
\end{abstract}
\maketitle

\section{Introduction}
In this paper, we prove an index theorem for operators on end-periodic manifolds, generalizing the index theorem of Atiyah, Patodi, and Singer~\cite{aps:I}. 

The Atiyah--Patodi--Singer theorem applies to a first-order elliptic differential operator $A$ on a compact oriented manifold $Z$ with boundary $Y$ that has the form   
\begin{equation}\label{E:aps}
A = \sigma\left(\frac{\p}{\p \theta} - B\right)
\end{equation}

\medskip\noindent
on a collar neighborhood of $Y$. Here, $B$ is a self-adjoint elliptic operator on $Y$ and $\sigma$ is a bundle isomorphism. We orient $Y$ so that the outer normal vector $\p/\p \theta$ followed by the orientation of $Y$ gives the orientation on $X$ (this is different from the orientation convention of \cite{aps:I} hence the negative sign in \eqref{E:aps}). The theorem states that the index of $A$, with respect to a certain global boundary condition, is given by
\begin{equation}\label{E:apsindex}
\ind A\; =\; \int_Z \ii(A)\; -\; \frac 1 2\, (h_B + \eta_B (Y)).
\end{equation}

\medskip\noindent
In this formula, $\ii(A)$ is the local index form~\cite{atiyah-bott-patodi:heat,BGV} whose integral would give the index of $A$ if $Z$ were a closed manifold, $h_B = \dim \ker B$, and the $\eta$--invariant $\eta_B (Y)$ is the value at $s = 0$ of the meromorphic extension of the function
\[
\sum \;\sign \lambda\, |\lambda|^{-s}
\]

\medskip\noindent
defined for sufficiently large $\Re (s)$ by summing over the non-zero eigenvalues $\lambda$ in the (discrete and real) spectrum of $B$. 

Because of~\eqref{E:aps}, the operator $A$ extends to an operator (still denoted by $A$) on the non-compact manifold obtained from $Z$ by attaching a product end $\R_+ \times Y$. If $\ker B = 0$, the $L^2$--closure of $A$ is known to be Fredholm, and its index is again given by formula ~\eqref{E:apsindex}. With a proper interpretation of $\ind A$ as in \cite[Section 3]{aps:I}, this formula holds even when $\ker B \ne 0$ and the $L^2$--closure of $A$ fails to be Fredholm.

Manifolds with product ends are a special case of the end-periodic manifolds that we study in this paper. By an end-periodic manifold we mean an open Riemannian manifold with an end modeled on an infinite cyclic cover $\xtilde$ of a compact manifold $X$ associated with a primitive cohomology class $\gamma \in H^1 (X;\Z)$; the case of several ends, which plays an important role in Sections~\ref{S:rho} and~\ref{S:psc}, can be treated similarly. To be precise, such a manifold is of the form
\begin{equation}\label{E:zp}
\zp\; =\; Z \,\cup\,W_0\,\cup\,W_1\,\cup\,W_2\,\cup\ldots,
\end{equation}
where $W_k$ are isometric copies of the fundamental segment $W$ obtained by cutting $X$ open along a oriented connected submanifold $Y$ Poincar\'e dual to $\gamma$, and $Z$ is a smooth compact manifold with boundary $Y$. 

End-periodic elliptic operators on end-periodic manifolds were studied by Taubes \cite{taubes:periodic}. A fundamental example of such an operator would be an operator having the form~\eqref{E:aps} on a manifold with product end, and more general versions appear elsewhere in geometry~\cite{mazzeo-pollack:gluing,mazzeo-pollack-uhlenbeck:yamabe,miller,MRS}; see also the discussion at the end of Section \ref{S:mflds}. Taubes established conditions under which the $L^2$--closure of an end-periodic elliptic operator is Fredholm, and calculated the index of the anti-self-dual operator occurring in Yang--Mills theory; this naturally raises the question of evaluating the index in the general case.

In this paper we present an index theorem for certain end-periodic operators, generalizing the Atiyah-Patodi-Singer index theorem. Assume that $\zp$ is even dimensional, and let $\Sm = \Sm^+\,\oplus\,\Sm^-$ be an end-periodic $\Z/2$--graded Dirac bundle as in~\cite{lawson-michelson} with associated chiral Dirac operator 
\[
\D^+(\zp): C^{\infty}(\zp;\Sm^+) \to  C^{\infty}(\zp;\Sm^-).
\]
Typical examples would include the spin Dirac operator on a spin manifold and the signature operator; more generally, either of these operators twisted by a complex vector bundle with unitary connection would give a Dirac operator to which our theorem would apply.

To state our theorem, write $\gamma = [\,df]$ for a choice of smooth function $f: \tilde X \to \mathbb R$ lifting a circle-valued function on $X$. According to Taubes \cite[Lemma 4.3]{taubes:periodic} the $L^2$--closure of $\D^+(\zp)$ is Fredholm if and only if the operators 
\[
\D^+_z = \D^+(X) - \ln z\cdot df
\]
on the closed manifold $X$ obtained by Fourier--Laplace transform are invertible on the unit circle $|z| = 1$. As a consequence of this condition, the operator $\D^+(X)$ has index zero and hence its index form $\ii(\D^+(X))$ is exact. Our end-periodic index theorem is now as follows.

\begin{thm}\label{T:aps}
Suppose that the $L^2$--closure of the operator $\D^+(\zp)$ is Fredholm, and choose a form $\omega$ on $X$ such that $d\omega = \ii(\D^+(X))$. Then 
\smallskip
\begin{equation}\label{E:indexintro}
\ind \D^+ (\zp) = \int_Z \ii (\D^+(Z))\; -
\int_Y\omega + \int_X df\wedge\omega\; - \frac 1 2\,\eta (X),
\end{equation}
where
\begin{equation}\label{E:etaintro}
\eta (X) = \frac 1 {\pi i}\,\int_0^{\infty}\oint_{|z| = 1}\;
\Tr \left(df\cdot \D^+_z \exp (-t \D^-_z \D^+_z)\right)\,\frac {dz} z\,dt.
\end{equation}
\end{thm}

\medskip
We will refer to \eqref{E:etaintro} as the \emph{periodic $\eta$-invariant}. The form $\omega$ is known in the literature as the transgressed class, cf. Gilkey \cite{gilkey}. 

In the product end case, one can choose $X = S^1\times Y$ so that $\tilde X = \mathbb R \times Y$, and let $f: \tilde X \to \mathbb R$ be the projection onto the first factor. With our orientation conventions, $\D^+(\zp)$ is of the form \eqref{E:aps} where $B = -\D$ and $\D$ is the self-adjoint Dirac operator on $Y$; see Nicolaescu \cite[page 70]{nic:israel}. One can easily see that in this case the operators $\D^+_z$ are invertible on the unit circle if and only if $\ker \D = 0$. We then show in Section \ref{S:product} that $\eta (X) = \eta_{\,\D} (Y)$.  A similar `spectral' interpretation of \eqref{E:etaintro} in general end-periodic case will be given in Section \ref{S:end-eta}.

How restrictive the Fredholmness condition in Theorem \ref{T:aps} is varies from one operator to another. For instance, the $L^2$--closure of the spin Dirac operator on a spin manifold $\zp$ of dimension at least 4 is Fredholm for a generic end-periodic metric, provided a certain topological obstruction vanishes; see \cite{ruberman-saveliev:periodic}. On the other hand, the $L^2$--closure of the signature operator on $\zp$ is never Fredholm. This explains the need to extend our Theorem \ref{T:aps} to the situations when the $L^2$--closure of $\D^+(\zp)$ is not Fredholm. We will discuss such an extension in Section \ref{S:non}, for a properly interpreted index analogous to the one treated in~\cite[\S3]{aps:I}.

It is worth mentioning that the index theorem of Atiyah, Patodi and Singer is just one of multiple index theorems generalizing the original Atiyah-Singer index theorem \cite{as:III} from compact to non-compact manifolds. In many of these generalizations, the operator is no longer required to be Fredholm but the index is interpreted through some kind of averaging procedure; examples include Atiyah's index theorem for coverings \cite{atiyah} and Roe's index theorem for certain open manifolds \cite{roe1,roe2}. Despite the common use of Fourier transform methods, our theorem is of a different nature; in fact, it is nearest to the classical case in that the operator under consideration is actually Fredholm, and its index is given by a formula similar to that of Atiyah, Patodi and Singer \cite{aps:I} with a different correction term.

%%%%%%%%%%%%%%%%%%%%%%%%%%%%%%%%%%%%%%%%%%%%%%%%%

\subsection{An outline of the proof}
Our proof of Theorem \ref{T:aps} is an adaptation to the end-periodic case of Melrose's proof of Atiyah--Patodi--Singer theorem \cite{melrose:aps}. Let $\D$ be the full Dirac operator on the end-periodic manifold $\zp$. The operator $\exp(-t \D^2)$ with $t > 0$ has a smoothing kernel $K(t;x,y)$. Unlike in the compact case, this does not mean that $\exp(-t \D^2)$ is trace class because $\tr(K(t;x,x))$ need not be an integrable function of $x \in Z_{\infty}$. To rectify this problem, we define in Section \ref{S:b-trace} a regularized trace $\btr$, and show that the associated supertrace 
\[
\sbtr(\exp(-t\D^2)) = \btr(\exp(-t\D^-\D^+)) - \btr(\exp(-t\D^+\D^-))
\]
has the desired properties that 
\[
\lim_{t\to 0}\;\, \sbtr(\exp(-t\D^2)) = \int_{Z}\ii(\D^+(Z))
\]
and 
\[
\lim_{t\to\infty}\; \sbtr(\exp(-t\D^2)) = \ind \D^+(\zp).
\]
In the closed case, the analysis shows that the supertrace of $\exp(-t\D^2)$ is constant in $t$, and this fact proves the index theorem. In our case, however, an easy calculation shows that
\[
\frac d{dt}\sbtr(\exp(-t\D^2)) = - \btr\left[\D^-,\D^+\exp(-t\D^-\D^+)\right],
\]
where the term on the right need not vanish because of the failure of the regularized trace $\btr$ to be a true trace functional, that is, to vanish on commutators. Integrating the latter formula with respect to $t \in (0,\infty)$, we obtain an index theorem with the ``defect'' in the form
\[
\int_0^{\infty}\btr\left[\D^-,\D^+\exp(-t\D^-\D^+)\right]\,dt.
\]
Expressing this integral in terms of the periodic $\eta$-invariant \eqref{E:etaintro} completes the proof of Theorem \ref{T:aps}.

There is a fair amount of analytic results behind each of the steps in the proof. Some of these results, like the short-time estimates on the kernel of $\exp(-t\D^2)$, are well known and hold for all manifolds of bounded geometry, of which end-periodic manifolds are a special case. Other results, like the convergence $\sbtr(\exp(-t\D^2)) \to \ind \D^+(\zp)$ as $t\to \infty$, are more delicate and rely on the original gradient estimates which are specific to manifolds with periodic ends. All of this analysis is collected in Section \ref{S:estimates}.

%%%%%%%%%%%%%%%%%%%%%%%%%%%%%%%%%%%%%%%%%%%%%%%%%%

\subsection{Calculations and applications}
Because the definition \eqref{E:etaintro} of the periodic $\eta$--invariant is so complex, it is not easy to make calculations beyond the product case. We present in Section \ref{S:inoue} a partial calculation for one family of examples, that of Dirac operators on certain Inoue surfaces~\cite{inoue}. These are important examples in the theory of non-K\"ahler complex surfaces; see for instance \cite{okonek-teleman:b+0} and \cite{teleman:definite}. Topologically they fiber over the circle with $3$-torus fiber, but the resulting infinite cyclic cover is not a metric product so $\eta(X)$ is not {\em a priori} related to the $\eta$--invariant of the fiber.

%
%\begin{thm}\label{T:inoue-intro}
%Let $X$ be an Inoue surface of class $S_M$ with $H_1(X;\Z) = \Z$, equipped with the Tricerri metric, and $\eta(X)$ the periodic $\eta$-invariant of the spin Dirac operator $\D^+(X)$. Then $\eta(X) =0$.
%\end{thm}

In Sections ~\ref{S:rho} and~\ref{S:psc} we present a sample application of our index theorem in differential geometry. First, in parallel to the theory of~\cite{aps:II}, we show how our periodic $\eta$--invariant give rise to a (metric-dependent) invariant $\tw_\alpha$ of an even-dimensional manifold $X$ equipped with a primitive cohomology class $\gamma \in H^1(X;\Z)$ and a unitary representation $\alpha$ of $\pi_1 (X)$. Recall that the $\tw$--invariants of~\cite[Section 3]{aps:II} are defined for odd-dimensional manifolds $Y$; our $\tw$--invariants are equal to these when $X = S^1 \times Y$. We apply our $\tw$--invariants in Section~\ref{S:psc} to detect components of the moduli space $\mmp(X)$ of Riemannian metrics of positive scalar curvature, modulo diffeomorphism.  Theorem~\ref{T:infinite} gives a somewhat more general version of the following result; cf. Botvinnik--Gilkey~\cite{botvinnik-gilkey:bordism}.

\begin{thm}\label{T:infinite-intro}
Let $Y$ be a closed connected spin manifold of dimension $4n-1$ with $n> 1$ and with a non-trivial finite fundamental group, and let $M$ be a closed spin manifold of dimension $4n$. If $Y$ and $M$ admit metrics of positive scalar curvature then $\pi_0(\mmp(S^1 \times Y)\conn M)$ is infinite.
\end{thm}

The topological techniques (surgery and handlebody theory) that go into the proof of this theorem are not completely available in low dimensions and so we are not able to obtain the result as stated if $n=1$.  By careful use of available techniques, in Theorem~\ref{T:4-dim}, we find $4$-manifolds of the form $X = (S^1 \times Y)\conn m\cdot (\sss)$ for which $\pi_0(\mmp(X))$ can be arbitrarily large.

%%%%%%%%%%%%%%%%%%%%%%%%%%%%%%%%%%%%%%%%%%%%%%%%

\subsection{Organization}
The paper is organized as follows. We begin in Section \ref{S:periodic} by reviewing the basics of the theory of end-periodic operators and by deriving an explicit formula for the smoothing kernel of the operator $\exp(-t\D^2)$ on the periodic manifold $\tilde X$. Analytic estimates on the heat kernel, despite being crucial to the proof, are postponed until Section \ref{S:estimates} for the sake of exposition. The regularized trace $\btr$ is defined in Section \ref{S:b-trace}. In Section \ref{S:trace} we derive the commutator trace formula for $\btr$; this is followed by the proof of Theorem \ref{T:aps} in Section \ref{S:proof}. Section \ref{S:end-eta} discusses a spectral  interpretation of the invariant $\eta(X)$, as well as its interpretation in terms of the von Neumann trace. We also calculate the periodic $\eta$--invariant for product manifolds and make partial calculations for certain Inoue surfaces. Theorem \ref{T:aps} is extended to the non-Fredholm case in Section~\ref{S:non}. In the same section we discuss the dependence of the periodic $\eta$--invariants on orientations and its relation to spectral flow. The periodic $\tw$--invariants are introduced in Section \ref{S:rho} and are applied to the study of metrics of positive scalar curvature in Section ~\ref{S:psc}. Several analytic results used elsewhere in the paper are collected in the final Section \ref{S:estimates}.
%and \ref{S:fourier}.
 
\medskip\noindent
\textbf{Acknowledgments:} We thank Boris Botvinnik, Peter Gilkey, Lev Kapitanski, Leonid Parnovski, and Andrei Teleman for generously sharing their expertise. We are also thankful to Pierre Albin, Gilles Carron, and Rafe Mazzeo for illuminating discussions of this material and of their approach to the index theorem in this setting.

%%%%%%%%%%%%%%%%%%%%%%%%%%%%%%%%%%%%%%%%%%%%%%%%

\section{End-periodic operators}\label{S:periodic}
We begin by describing manifolds with periodic ends, and the class of operators that we will consider.  We will restrict ourselves to the situation with one end; the extension to several ends is routine.  

%%%%%%%%%%%%%%%%%%%%%%%%%%%%%%%%%%%%%%%%%%%%%%%%

\subsection{End-periodic manifolds and operators}\label{S:mflds}
Let $X$ be an oriented compact manifold endowed with a primitive cohomology class $\gamma \in H^1(X;\Z)$. This data gives rise to an infinite cyclic covering $p: \xtilde \to X$, together with a generator $T$ of the covering translations, which we will often denote by $T(x) = x+1$. Choose a smooth function $X \to S^1$ which pulls back the generator of $H^1(S^1;\Z)$ to $\gamma$.  We fix a lift $f: \xtilde \to \R$ of this function; it has the property that $f(x+1) = f(x) + 1$.  Note that while $f$ does not descend to a real-valued function on $X$, its differential does, and we will abuse notation by viewing $df$ as a $1$-form on $X$.

Choose an oriented, connected submanifold $Y \subset X$ that is Poincar\'e dual to $\gamma$, and cut $X$ open along $Y$ to obtain a cobordism $W$ with boundary $\p W = -Y\,\cup\,Y$. Note that 
\[
\xtilde\; =\; \bigcup_{k= -\infty}^\infty W_k,
\]
where $W_k$ are just copies of $W$. By definition, an \emph{end-periodic manifold} with end modeled on the infinite cyclic cover of $X$ is a manifold of the form 
\begin{equation}\label{E:zinfty}
\zp\; =\; Z\,\cup\, W_0\,\cup\, W_1\, \cup\, W_2\, \cup \cdots,
\end{equation}
where $Z$ is a smooth oriented compact manifold with boundary $Y$. There are obvious definitions of Riemannian metrics, bundles, and differential operators in this setting; in short, all data over the end should be pulled back from the same sort of data on $X$. 

We will largely follow the notation of~\cite{lawson-michelson} for index-theoretic notions. We consider Dirac operators $\D(M)$ defined on sections of a Dirac bundle $\Sm$ over a manifold $M$. When the dimension of $M$ is even, the Dirac bundle $\Sm = \Sm^+ \oplus \Sm^-$ is $\Z/2$--graded, and the Dirac operator $\D(M)$ decomposes into the chiral Dirac operators,
\[
\D(M)\; =\; 
\begin{pmatrix}
0& \D^-(M)\\
\D^+(M) &0
\end{pmatrix}
\]

\smallskip\noindent
with $\D^\pm(M): C^\infty(M,\Sm^\pm) \to C^\infty(M,\Sm^\mp)$. Note that part of the data of the Dirac bundle is a connection on $\Sm$ compatible with Clifford multiplication and with the grading.

There are many examples of end-periodic manifolds, and our analysis will apply to the natural geometric operators on these manifolds, such as the spin Dirac operator on a spin manifold and the signature operator, as well as to their twisted versions. We remark that an end may be topologically a product but not geometrically so; examples of this sort are warped products $\R \times Y$ with metric $d\theta^2 + \phi(\theta)\,g^Y$ (with periodic warping function $\phi(\theta)$). Manifolds with periodic ends that are not topologically products also abound; for instance, the manifold $X$ obtained by ($0$--framed for $n =3$) surgery on a knot in $S^n$ will have infinite cyclic cover $\tilde X$ that is not a product if the Alexander polynomial of the knot is not monic. A typical end-periodic manifold with end modeled on this $\tilde X$ may be obtained by cutting $\tilde X$ along a lift of a Seifert surface for the knot, and filling in this Seifert surface by an $n$-manifold that it bounds.  In a recent paper~\cite{mrowka-ruberman-saveliev:derham}, we analyzed the de Rham complex (with weights, as in Section~\ref{S:non}) on end-periodic manifolds.  When the end is modeled on the infinite cyclic cover of surgery on a knot, the Alexander module of the knot determines the behavior of the index as the weights are varied.

%%%%%%%%%%%%%%%%%%%%%%%%%%%%%%%%%%%%%%%%%%%%%%%

\subsection{Fredholm theory for end-periodic operators}\label{S:fred}
We briefly review relevant parts of the Fredholm theory of end-periodic operators following Taubes~\cite{taubes:periodic}, starting with the definition of the weighted Sobolev spaces.

Given $\delta \in \R$ and a non-negative integer $k$, we will say that $u \in L^2_{k,\delta}\, (\zp,\Sm)$ if and only if $e^{\delta f} u \in L^2_k\,(\zp,\Sm)$, where $f: \zp \to \mathbb R$ is an extension of $f$ to $\zp$. We define
\[ 
\| u \|_{L^2_{k,\delta}\,(\zp,\Sm)} = \|\,e^{\delta f}\,u \|_{L^2_k\,(\zp,\Sm)}.  
\] 
Assume for the sake of concreteness that $\zp$ is even dimensional. Then as usual, the operator $\D^+(\zp)$ extends to a bounded operator 
\begin{equation}\label{E:sobolev}
\D^+(\zp): L^2_{k+1,\delta}\, (\zp,\Sm^+) \to L^2_{k,\delta}\, (\zp,\Sm^-),
\end{equation}
and similarly for $\D^-(\zp)$. An excision principle shows that the operator \eqref{E:sobolev} is Fredholm if and only if the operator $\D^+(\xtilde): L^2_{k+1,\delta}\, (\xtilde,\Sm^+) \to L^2_{k,\delta}\, (\xtilde,\Sm^-)$ is Fredholm (or equivalently, invertible). Taubes gives a Fredholmness criterion using the Fourier-Laplace transform as follows.

Given a spinor $u \in C^{\infty}_0 (\tilde X;\Sm)$ and a complex number $z \in \C^*$, the \emph{Fourier--Laplace transform} of $u$ is defined as 
\[ 
\hat u_z(x) = z^{\,f(x)}\,\sum_{n = -\infty}^{\infty}  z^{\,n}\, u(x+n),
\] 
for a fixed branch of $\ln z$. Since $u$ has compact support, the above sum is finite. One can easily check that $\hat u_z(x + 1) =  \hat u_z (x)$ for all $x \in \tilde X$. Therefore, for every $z \in \C^*$, we have a well defined spinor $\hat u_z$ over $X$ that depends analytically on $z$. The spinor $u$ can be recovered from its Fourier--Laplace transform using the formula 
\smallskip
\begin{equation}\label{E:zinv}
u(x)\; =\; \frac 1 {2\pi i}\,\oint_{|z| = 1}\; z^{-f(x)}\;\hat u_z(x_0)\,
\frac {dz}z, 
\end{equation} 
where $p:\tilde X \to X$ is the covering projection,  $x_0 = p(x) \in X$, and the contour integral is taken counterclockwise. This can be checked by direct substitution. 

The Fourier-Laplace transform extends to the weighted Sobolev spaces defined above. Conjugating the operators $\D^{\pm}(\xtilde)$ by the Fourier--Laplace transform, we obtain holomorphic families of twisted Dirac operators on $X$,
\begin{equation}\label{E:holo}
\D^{\pm}_z (X)\; =\; \D^{\pm}(X)\; -\; \ln z\cdot df,\quad z \in \C^*.
\end{equation}

\begin{proposition}[{\cite[Lemma 4.3]{taubes:periodic}}]\label{P:taubes}
The operator \eqref{E:sobolev} is Fredholm if and only if the operators 
$\D^+_z(X)$ are invertible for all $z$ on the circle $|z| = e^\delta$.
\end{proposition}

\begin{corollary}\label{C:fred}
A necessary condition for the operator \eqref{E:sobolev} to be Fredholm is that $\ind \D^+(X) = 0$. 
\end{corollary}

Given that $\ind \D^+ (X) = 0$, the set of points $z \in \C^*$ where the operators $\D^+_z(X)$ or, equivalently, $\D^-_z(X)$, are not invertible will be referred to as the \emph{spectral set} of the family $\D^+_z(X)$. The following result is due to Taubes \cite[Theorem 3.1]{taubes:periodic}.

\begin{theorem}\label{T:fred1}
Suppose that $\ind \D^+ (X) = 0$ and the map $df: \ker \D^+ (X) \to \coker \D^+ (X)$ given by Clifford multiplication by $df$ is injective. Then the spectral set of the family $\D^+_z (X)$ is a discrete subset of $\C^*$. In particular, the operator \eqref{E:sobolev} is Fredholm for all but a discrete set of $\delta \in \R$.
\end{theorem}

\begin{remark}
In the paper \cite{ruberman-saveliev:periodic} the second and third authors investigated the special case of the spin Dirac operator on end-periodic manifolds of dimension at least 4 and gave a condition that guarantees that the spectral set of the family $\D^+_z (X)$ is both discrete and avoids the unit circle $|z| = 1$ for a generic metric on $X$. If the dimension of $X$ is divisible by four, that condition is simply the necessary condition of Corollary \ref{C:fred}.
\end{remark}

%%%%%%%%%%%%%%%%%%%%%%%%%%%%%%%%%%%%%%%%%%%%%%%%

\subsection{Smoothing kernels}
This section introduces smoothing kernels of operators of the form $h(\D)$, where $h$ is a rapidly decaying functions. To avoid complicating the exposition, we have delayed a detailed discussion of relevant analytic matters until Section \ref{S:estimates}. 

Let $\D$ be a Dirac operator on a manifold $M$, which in all of our applications will be either closed or end-periodic. For any rapidly decaying function $h: \R \to \R$, the operator $h(\D)$ defined using the spectral theorem can be written as 
\begin{equation}\label{E:h(D)}
h(\D)(u) (x) = \int_M\, K(x,y)\,u(y)\;dy,
\end{equation}

\smallskip\noindent
where $K(x,y)$ is a smooth section of the bundle $\Hom\,(\pi_R^* \Sm,\pi_L^* \Sm)$, and $\pi_L$ and $\pi_R$ are projections onto the two factors of $M \times M$. We refer to $K(x,y)$ as the \emph{smoothing kernel}. For example, the operators $\exp(-t \D^2)$ and $\D\,\exp(-t \D^2)$ are represented by such smoothing kernels for all $t > 0$. When $M$ is even dimensional, we can restrict our operators to the sections of $\Sm^{\pm}$ to obtain their chiral versions like $\exp(-t \D^-\D^+)$ or $\D^+ \exp(-t \D^-\D^+)$, which again are represented by smoothing kernels.

In the original proof of the Atiyah--Patodi--Singer index theorem~\cite{aps:I}, a crucial role was played by an explicit formula for the smoothing kernel of the operator $\exp(-t \D^2)$ on a cylinder obtained from the classical solution of the heat equation \cite[Section 2]{aps:I}. Such an explicit formula is not available in our more general setting, however, we present below an equally useful formula for the smoothing kernel of the operator $h(\D)$ on the infinite cyclic cover $\tilde X$ in terms of data on $X$.

Let $\D_z = \D (X) - \ln z\cdot df$, $z \in \mathbb C^*$. The smoothing kernels of operators $h(\D)$ on $\tilde X$ and of $h(\D_z) = h(\D)_z$ on $X$ will be called $\tilde K(x,y)$ and $K_z(x,y)$, respectively. The following proposition, which expresses $\tilde K$ in terms of $K_z$, can be verified by a direct calculation with the Fourier--Laplace transform.

\begin{proposition}
Let $p:\tilde X \to X$ be the covering projection then, for any $x, y \in \tilde X$ and $x_0 = p(x)$, $y_0 = p(y)$, we have
\begin{equation}\label{E:kkz}
\tilde K(x,y)\; =\; \frac 1 {2\pi i}\,\oint_{|z| = 1} z^{f(y)-f(x)} 
K_z (x_0,y_0)\;\frac{dz}z.
\end{equation}
\end{proposition}

\begin{remark}
Note that formula \eqref{E:kkz} implies that the smoothing kernel $\tilde K(x,y)$ is periodic in that 
\begin{equation}\label{E:per}
\tilde K(x+k,y+k) = \tilde K(x,y)\quad\text{for any $k \in \Z$}.
\end{equation}
In addition, if $\w K_z (x_0,y)$ is the Fourier--Laplace transform of $\tilde K(x,y)$ with respect to the variable $x$ then
\begin{equation}\label{E:wkz}
\w K_z (x_0,y) = z^{f(y)}\, K_z (x_0,y_0).
\end{equation}
\end{remark}

\smallskip

%%%%%%%%%%%%%%%%%%%%%%%%%%%%%%%%%%%%%%%%%%%%%%%%%%%%%%%%%%%%%%%%%%%%%%%%%%%%

\section{Regularized trace}\label{S:b-trace}
A smoothing operator \eqref{E:h(D)} need not be trace class on a non-compact manifold because the integral of $\tr (K(x,x))$, which is used to define the operator trace, may diverge. Such an integral can be regularized in many different ways; the regularization we choose is inspired by that of Melrose \cite{melrose:aps}. It applies to end-periodic manifolds like $\zp$ and the operators $\D^{\,m} \exp(-t \D^2)$, $m \ge 0$, and their chiral versions such as $ \D^- \D^+\exp(-t \D^-\D^+)$. 

Extending this construction to a larger class of operators, while it may be an interesting problem in its own right, is certainly beyond the scope of this paper.

%%%%%%%%%%%%%%%%%%%%%%%%%%%%%%%%%%%%%%%%%%%%%%%%%%%%%%%%%%%%%%%%%%%%%%%%%%%%

\subsection{Definition of the regularized trace}
Let us fix an integer $m \ge 0$ and consider the operator $\D^{\,m} \exp(-t\D^2)$. This operator will be called $P$ or $\tilde P$, depending on whether $\D$ is the Dirac operator on $\zp$ or $\tilde X$. The smoothing kernels of $P$ and $\tilde P$ will be denoted by $K(t;x,y)$ and $\tilde K(t;x,y)$, respectively. 

Let $Z_N = Z\,\cup\,W_0\,\cup \ldots \cup\,W_N$ for any integer $N \ge 0$. We define the \emph{regularized trace} of $P$ by the formula
\begin{multline}\label{E:btr}
\btr (P) = \\ \lim_{N \to \infty} \left[ \int_{Z_N} \tr\,(K(t;x,x))\,dx - 
(N + 1) \int_{W_0} \tr\,(\tilde K(t;x,x))\,dx \right].
\end{multline}

\medskip

\begin{lemma}
For any $t > 0$, the limit \eqref{E:btr} exists.
\end{lemma}

\begin{proof}
Write
\begin{multline}\notag
\int_{Z_N} \tr\,(K(t;x,x))\,dx = \int_Z \tr\,(K(t;x,x))\,dx\, +
\sum_{k = 0}^N\; \int_{W_k} \tr\,(K(t;x,x))\,dx = \\ 
\int_Z \tr\,(K(t;x,x))\,dx + \sum_{k = 0}^N\; \int_{W_0} \tr\,(K(t;x+k,x+k) 
- \tilde K(t;x+k,x+k))\,dx \\ + (N+1) \int_{W_0} \tr\,(\tilde K(t;x,x))\,dx,
\end{multline}
where we used \eqref{E:per} in the last line. Use Corollary \ref{C:diff} and Remark \ref{R:diff} to estimate
\begin{equation}\label{E:btr1}
\int_{W_0} |\tr\,(K(t;x+k,x+k) - \tilde K(t;x+k,x+k))|\,dx\;\le\;
C_1\,e^{-C_2\,(k-1)^2/t}
\end{equation}

\medskip\noindent
for all $k \ge 1$. Therefore, the series 
\begin{equation}\label{E:btr2}
\sum_{k = 0}^{\infty}\;\int_{W_0} \tr\,(K(t;x+k,x+k)-\tilde K(t;x+k,x+k))\,dx
\end{equation}
converges absolutely for any $t > 0$. This completes the proof. 
\end{proof}

\begin{lemma}\label{L:btr}
The regularized trace $\btr(P)$ is a continuous function of $t\in (0,\infty)$.
\end{lemma}

\begin{proof}
Since both $K(t;x,x)$ and $\tilde K(t;x,x)$ are continuous functions of $t \in (0,\infty)$; see Section \ref{S:smooth} and \cite[Proposition 2.10]{roe1}, the result follows from uniform convergence of series \eqref{E:btr2} on bounded intervals.
\end{proof}

%%%%%%%%%%%%%%%%%%%%%%%%%%%%%%%%%%%%%%%%%%%%%%%%%%%%%%%%%%%%%%%%%%%%%%%%%%

\subsection{A formula for the regularized trace}
The formula \eqref{E:kkz} can be used to express the correction term in formula \eqref{E:btr} for the regularized trace in terms of the family $\D_z = \D_z(X)$ of twisted Dirac operators on the closed manifold $X$. Restricting \eqref{E:kkz} to the diagonal $x = y$ and applying the matrix trace, we obtain 
\[
\tr\,(\tilde K(t;x,x)) = \frac 1 {2\pi i}\,\oint_{|z| = 1} 
\tr\,(K_z (t;x_0,x_0))\;\frac{dz}z,
\]
where $x_0 = p(x)$, $y_0 = p(y)$, and $K_z$ is the smoothing kernel of the operator $P_z = \D_z^{\,m} \exp(-t \D_z^2)$ on $X$. Since $X$ is closed, $P_z$ is of trace class with
\[
\Tr (P_z) = \int_X \tr\,(K_z (t;x_0,x_0))\,dx_0.
\]
Therefore, we can write
\begin{equation}\label{E:rtr}
\btr(P) = \lim_{N \to \infty} \left[ \int_{Z_N} \tr\,(K(t;x,x))\,dx\; -\; 
\frac {N + 1}{2\pi i}\; \oint_{|z| = 1} \Tr (P_z)\;\frac{dz}z\right]
\end{equation}

\smallskip

%%%%%%%%%%%%%%%%%%%%%%%%%%%%%%%%%%%%%%%%%%%%%%%%%%%%%%%%%%%%%%%%%%%%%%%%%%%

\section{A commutator trace formula}\label{S:trace}
Let $P$ and $Q$ be the chiral versions of the operator $\D\,\exp(-t\D^2)$ on $\zp$ and define $\btr [P,Q] = \btr (PQ) - \btr (QP)$. The purpose of this section is to derive a formula for $\btr [P,Q]$ solely in terms of data on $X$. This formula will be the main ingredient in the proof of index theorem in Section \ref{S:proof}.

%%%%%%%%%%%%%%%%%%%%%%%%%%%%%%%%%%%%%%%%%%%%%%%%%%%%%%%%%%%%%%%%%%%%%%%%%%%

\subsection{First step}\label{S:red1}
Let us fix $s > 0$ and $t > 0$ and consider the operators $P = \D^-\exp(-s \D^+\D^-)$ and $Q = \D^+\exp(-t \D^-\D^+)$ on $\zp$. Note that the regularized traces $\btr (PQ)$ and $\btr (QP)$ are well defined because both operators $PQ$ and $QP$ are of the type described in Section \ref{S:b-trace}. This follows from the identity $\D^{\mp}\exp(-t \D^{\pm}\D^{\mp}) = \exp(-t \D^{\mp}\D^{\pm}) \D^{\mp}$ which is easily verified using \eqref{E:one}. Write
\[
(P u)(x) = \int_{\zp} K_P (x,y) u(y)\,dy,\quad
(Q u)(x) = \int_{\zp} K_Q (x,y) u(y)\,dy,
\]
then
\[
(PQ)u(x) = \int_{\zp} \left(\;\int_{\zp} K_P (x,y)\, K_Q (y,z)\,dy\;\right)
\,u(z)\,dz,
\]
so that
\[
K_{PQ}(x,x) = \int_{\zp} K_P (x,y) K_Q (y,x)\,dy.
\]
Similarly,
\[
K_{QP}(x,x) = \int_{\zp} K_Q (x,y) K_P (y,x)\,dy.
\]
For any fixed $N$,
\begin{alignat*}{1}
\int_{Z_N} (\tr K_{PQ})(x,x)\,dx &= 
\iint_{Z_N\times \zp} \tr\, (K_P (x,y) K_Q (y,x))\,dx\,dy
\end{alignat*}
and 
\begin{alignat*}{1}
\int_{Z_N} (\tr K_{QP})(x,x)\,dx &= 
\iint_{Z_N\times \zp} \tr\, (K_Q (x,y) K_P (y,x))\,dx\,dy \\ &=
\iint_{\zp\times Z_N} \tr\, (K_P (x,y) K_Q (y,x))\,dx\,dy.
\end{alignat*}

\smallskip

The Gaussian estimates \eqref{E:short} ensure that the double integrals in the last 
two formulas are absolutely convergent and in particular that changing the order of 
integration is justified. 
Therefore, we can write
\begin{multline}\label{E:star}
\int_{Z_N} (\tr K_{PQ})(x,x)\,dx - \int_{Z_N} (\tr K_{QP})(x,x)\,dx \\ =
\iint_{\Delta_+} \tr\, (K_P (x,y) K_Q (y,x))\,dx\,dy \\ -
\iint_{\Delta_-} \tr\, (K_P (x,y) K_Q (y,x))\,dx\,dy,
\end{multline}

\smallskip\noindent
where $\Delta_+ = Z_N \times (\zp - Z_N)$ and $\Delta_- = (\zp - Z_N) \times Z_N$ are shown schematically in Figure \ref{fig1}. Note that $\Tr [P,Q]_z = \Tr [P_z,Q_z] = 0$ on the closed manifold $X$ hence the correction term in the expression for $\btr [P,Q]$ vanishes and $\btr [P,Q]$ is obtained by simply passing to the limit as $N \to \infty$ in \eqref{E:star}.

\begin{figure}[!ht]
\centering
\psfrag{k}{$k$}
\psfrag{l}{$\ell$}
\psfrag{D+}{$\Delta_+$}
\psfrag{D-}{$\Delta_-$}
\psfrag{N}{$N$}
\psfrag{N+1}{$N+1$}
\includegraphics{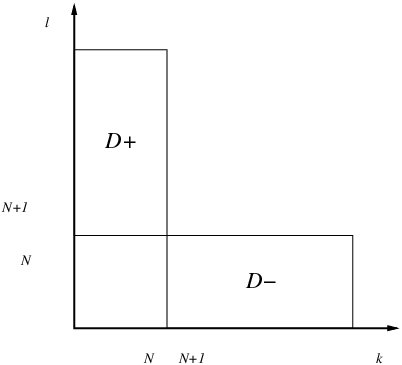}
\caption{}
\label{fig1}
\end{figure}

Before we go on, we will introduce some notation. Given two numerical sequences 
$A_N$ and $B_N$, we will write $A_N \eq B_N$ to mean that $A_N - B_N \to 0$ 
as $N \to \infty$. 

\begin{lemma}\label{L:4.2}
\begin{multline}\notag
\iint_{\Delta_-} \tr\, (K_P (x,y) K_Q (y,x))\,dx\,dy \\
\eq \;\sum_{\ell = 0}^N\; \sum_{k = N+1}^{\infty} \;\iint_{\,W_k\times W_{\ell}}
\tr\, (K_P (x,y) K_Q (y,x))\,dx\,dy.
\end{multline}
\end{lemma}

\smallskip

\begin{proof}
The difference between the two sides of the equation is the integral 
\[
\iint_{(\zp - Z_N)\times Z} \tr\, (K_P (x,y) K_Q (y,x))\,dx\,dy,
\]
whose absolute value can be estimated using \eqref{E:short} by a multiple of
\[
\sum_{k = N+1}^{\infty}\;\iint_{W_k \times Z}\;e^{-C k^2/t}\,dx\,dy\; \to\, 0
\quad\text{as\; $N\; \to \infty$}.
\]
\end{proof}

Change the variables of summation from $k$ and $\ell$ to $m = k - \ell$ 
and $k$, then the right hand side of the equation in Lemma \ref{L:4.2} 
becomes
\begin{multline}\label{E:sum1}
\sum_{m = 1}^N\; \sum_{k = N+1}^{N+m}\;\iint_{\,W_k\times W_{k-m}}
\tr\, (K_P (x,y) K_Q (y,x))\,dx\,dy \\ 
+ \sum_{m = N+1}^{\infty}\; \sum_{k=m}^{N+m}\;\iint_{\,W_k\times W_{k-m}}
\tr\, (K_P (x,y) K_Q (y,x))\,dx\,dy.
\end{multline}

\medskip\noindent
With the help of this formula, we will trade the kernels of $P$ and $Q$ on 
$\zp$ for the kernels of $\tilde P$ and $\tilde Q$ on $\tilde X$, at the 
expense of adding a term that approaches zero as $N \to \infty$. 

\begin{lemma}\label{L:4.3}
\begin{multline}\notag
\iint_{\Delta_-} \tr\, (K_P (x,y) K_Q (y,x))\,dx\,dy \\
\eq\; \sum_{m = 1}^{\infty}\; \sum_{k = N+1}^{N+m}\;\iint_{\,W_k\times W_{k-m}}
\tr\, (K_{\tilde P} (x,y) K_{\tilde Q} (y,x))\,dx\,dy.
\end{multline}
\end{lemma}

\begin{proof}
We begin by noting that the second term in \eqref{E:sum1} approaches zero 
as $N \to \infty$ because it can be estimated by a multiple of
\begin{multline}\notag
\sum_{m = N+1}^{\infty} (N+1)\,e^{-C (m-1)^2/t}\;\le\;
(N+1) \sum_{m = N+1}^{\infty} e^{-C (m-1)/t} \\ \le\;
(N+1)\, e^{-C N/t}\, \sum_{n = 0}^{\infty} e^{-C n/t} \; \to 0,
\end{multline}
where $n = m - (N+1)$ and $C$ is a positive constant. Similarly, we obtain
\begin{multline}\notag
\sum_{m = N+1}^{\infty}\; \sum_{k = N+1}^{N+m}\;\iint_{\,W_k\times W_{k-m}}
|\tr\, (K_{\tilde P} (x,y) K_{\tilde Q} (y,x))|\,dx\,dy \\ \le\;
C_0\,\sum_{m = N+1}^{\infty}\;m\,e^{-C(m-1)^2/t}\;\to 0\quad
\text{as\; $N \to \infty$},
\end{multline}
where $C_0$ and $C$ are positive constants. All that is left to estimate is
\[
\sum_{m = 1}^N\; \sum_{k = N+1}^{N+m}\;\iint_{\,W_k\times W_{k-m}}
\tr\, (K_P (x,y) K_Q (y,x) - K_{\tilde P} (x,y) K_{\tilde Q} (y,x))\,dx\,dy.
\]

\smallskip\noindent
Write 
\begin{multline}\notag
|K_P (x,y) K_Q (y,x) - K_{\tilde P} (x,y) K_{\tilde Q} (y,x)| \;\le \\
|K_P (x,y)||K_Q (y,x) - K_{\tilde Q} (y,x)| +
|K_P (x,y) - K_{\tilde P} (x,y)||K_{\tilde Q} (y,x)|
\end{multline}
and estimate the integrals of each of the summands on the right separately. 
By breaking the summation over $m$ into two parts we obtain  
\begin{multline}\notag
\sum_{m = 1}^{\lfloor N/2\rfloor}\; \sum_{k = N+1}^{N+m}\;\iint_{W_k\times W_{k-m}}
|K_P (x,y)||K_Q (y,x) - K_{\tilde Q} (y,x)|\,dx\,dy\\
\le \;C_0\,\sum_{m = 1}^{\lfloor N/2\rfloor}\; m\,e^{-C_1 (m-1)^2/s}\,e^{-C_2 N^2/t} 
\\ \le\; C_0\,e^{-C_2 N^2/t}\,\sum_{m = 1}^{\infty}\; m\,e^{-C_1 (m-1)^2/s}\;\to\; 0,
\end{multline}
using the estimate \eqref{E:short} for $|K_P (x,y)|$ and that 
from Proposition \ref{P:diff2} for $|K_Q (y,x)-K_{\tilde Q} (y,x)|$, and 
\begin{multline}\notag
\sum_{m = \lfloor N/2\rfloor+1}^N\; \sum_{k = N+1}^{N+m}\;\iint_{\,W_k\times W_{k-m}}
|K_P (x,y)||K_Q (y,x) - K_{\tilde Q} (y,x)|\,dx\,dy \\
\le\;C_0\,\sum_{m = \lfloor N/2\rfloor +1}^{\infty}\; m\,e^{-C_1 (m-1)^2(1/t + 1/s)}\;\to 0,
\end{multline}
using the estimate \eqref{E:short} for both $|K_P (x,y)|$ 
and $|K_Q (y,x)-K_{\tilde Q} (y,x)|$. The estimates for $|K_P (x,y) - 
K_{\tilde P} (x,y)||K_{\tilde Q} (y,x)|$ are similar. 
\end{proof}

By replacing $\Delta_-$ with $\Delta_+$ and proceeding exactly as above, we 
obtain the following result. 

\begin{lemma}\label{L:delta+}
\begin{multline}\notag
\iint_{\Delta_+} \tr\, (K_P (x,y) K_Q (y,x))\,dx\,dy \\ \eq\; 
\sum_{m = -1}^{-\infty}\; \sum_{\ell = N+1}^{N-m}\;\iint_{\,W_{m+\ell}\,\times W_{\ell}}
\tr\, (K_{\tilde P} (x,y) K_{\tilde Q} (y,x))\,dx\,dy.
\end{multline}
\end{lemma}

\medskip
The kernels $K_{\tilde P}$ and $K_{\tilde Q}$ have the property that 
$K_{\tilde P}(x+n,y+n) = K_{\tilde P} (x,y)$ for any integer $n$ and
similarly for $K_{\tilde Q}$; see \eqref{E:per}. Use this observation 
together with Lemma \ref{L:4.3}, Lemma \ref{L:delta+} and formula 
\eqref{E:star} to obtain
\begin{alignat*}{1}
\int_{Z_N} & (\tr K_{PQ})(x,x)\,dx - \int_{Z_N} (\tr K_{QP})(x,x)\,dx\;\eq \\
& \sum_{m = -1}^{-\infty}\; \sum_{\ell = N+1}^{N-m}\;\iint_{\,W_0\times W_0}
\tr\, (K_{\tilde P} (x+m,y) K_{\tilde Q} (y,x+m))\,dx\,dy\, - \\
& \quad\sum_{m = 1}^{\infty}\; \sum_{k = N+1}^{N+m}\;\iint_{\,W_0\times W_0}
\tr\, (K_{\tilde P} (x+m,y) K_{\tilde Q} (y,x+m))\,dx\,dy
\end{alignat*}

\smallskip\noindent
as $N \to \infty$. Next, observe that the integrands in the above 
formula do not depend on $k$ or $\ell$ hence this last formula can be 
written as follows
\begin{alignat*}{1}
\int_{Z_N} & (\tr K_{PQ})(x,x)\,dx - \int_{Z_N} (\tr K_{QP})(x,x)\,dx\;\eq \\
& - \sum_{m = -\infty}^{\infty}\;m\cdot\iint_{\,W_0\times W_0}
\tr\, (K_{\tilde P} (x+m,y) K_{\tilde Q} (y,x+m))\,dx\,dy.
\end{alignat*}

\noindent
Finally, pass to the limit as $N \to \infty$ in this formula to conclude that
\begin{multline}\label{E:star2}
\btr [P,Q] = \\ - \sum_{m = -\infty}^{\infty}\;m\cdot\iint_{\,W_0\times W_0}
\tr\, (K_{\tilde P} (x+m,y) K_{\tilde Q} (y,x+m))\,dx\,dy.
\end{multline}

\medskip

%%%%%%%%%%%%%%%%%%%%%%%%%%%%%%%%%%%%%%%%%%%%%%%%%%%%%%%%%%%%%%%%%%%%%%%%%%%

\subsection{Second step}\label{S:step2}
Our next task will be to calculate the expression 
\begin{equation}\label{E:kk}
- \sum_{m = -\infty}^{\infty}\;m\cdot K_{\tilde P} (x+m,y) K_{\tilde Q} (y,x+m),
\quad x, y \in W_0,
\end{equation}
in terms of the kernels $K_{P_z}$ and $K_{Q_z}$ of the holomorphic families
$P_z$ and $Q_z$. 

\begin{lemma} (Parseval's relation)\;
For any $x \in W_0$, we have 
\[
\sum_m\; u(x+m)\,v(x+m) \;=\; \frac 1 {2\pi i}\;
\oint_{|z| = 1} \w u_z (x)\;\w v_{1/z}(x)\;\frac {dz} z.
\]
\end{lemma}

\begin{proof}
According to the definition of the Fourier--Laplace transform, for any 
$u$ and $v$ with compact support, we have 
\[
\w u_z(x) = z^{f(x)}\,\sum_n \,z^n\,u(x+n)\quad \text{and}\quad 
\w v_{1/z}(x) = z^{-f(x)}\,\sum_m \,z^{-m}\,v(x+m). 
\]
Therefore, 
\[
\frac 1{2\pi i}\,\oint_{|z| = 1}\w u_z(x)\;\w v_{1/z}(x)\;\frac{dz}z\; 
=\;\frac 1{2\pi i}\,\sum_{m,n}\;\oint_{|z| = 1}z^{n-m}\,u(x+n)\,v(x+m)
\;\frac{dz}z,
\]
and the result obviously follows. 
\end{proof}

\begin{lemma}\label{L:kk}
For any $x \in W_0$, we have 
\[
\sum_m\; m\, u(x+m)\,v(x+m) \;=\; \frac 1 {2\pi i}\;
\oint_{|z| = 1} \frac {\p}{\p z}\left(z^{-f(x)}\,\w u_z (x)\right)\,
z^{f(x)}\,\w v_{1/z}(x)\,dz.
\]
\end{lemma}

\begin{proof}
The result follows as above by plugging into the contour integral the 
expressions
\[
z^{-f(x)}\, \w u_z(x) = \sum_n \,z^n\,u(x+n)\quad \text{and}\quad 
z^{f(x)}\, \w v_{1/z}(x) = \sum_m \,z^{-m}\,v(x+m). 
\]
\end{proof}

We will apply Lemma \ref{L:kk} with $u(x+m) = K_{\tilde P} (x+m,y)$ and 
$v(x+m) = K_{\tilde Q} (y,x+m)$, where $x, y \in W_0$. Comparing 
\[
u(x+m) = \frac 1 {2\pi i}\;\oint_{|z|=1} z^{-f(x)-m}\; \w u_z(x)\;\frac {dz} z
\]
with 
\[
K_{\tilde P} (x+m,y) = \frac 1 {2\pi i}\;\oint_{|z|=1} z^{f(y)-f(x)-m}\,
K_{P_z}(x,y)\;\frac {dz} z;
\]
see \eqref{E:kkz}, we obtain
\[
\w u_z (x) = z^{f(y)}\, K_{P_z}(x,y);
\]
see also \eqref{E:wkz}. Similarly, substitute $w = 1/z$ in 
\[
v(x+m) = \frac 1 {2\pi i}\;\oint_{|w|=1} w^{-f(x)-m}\; \w v_w(x)\;\frac {dw} w
\]
to obtain 
\[
v(x+m) = \frac 1 {2\pi i}\;\oint_{|z|=1} z^{f(x)+m}\; \w v_{1/z}(x)\;
\frac {dz} z.
\]
Comparing the latter with 
\[
K_{\tilde Q} (y,x+m) = \frac 1 {2\pi i}\;\oint_{|z|=1} z^{f(x)+m-f(y)}\, 
K_{Q_z}(y,x)\;\frac {dz} z;
\]
see \eqref{E:kkz}, we obtain
\[
\w v_{1/z} (x) = z^{-f(y)}\, K_{Q_z}(y,x). 
\]
Substitute the above in the formula of Lemma \ref{L:kk} to obtain the 
following formula for the expression \eqref{E:kk}\,:
\begin{multline}\notag
- \sum_m\;m\cdot K_{\tilde P} (x+m,y) K_{\tilde Q} (y,x+m) \\ 
= - \frac 1 {2\pi i}\,\oint_{|z| = 1}\;z^{f(x) - f(y)}\;\frac {\p}{\p z}
\left(z^{-f(x) + f(y)} K_{P_z}(x,y)\right) K_{Q_z}(y,x)\,dz.
\end{multline}
A direct calculation shows that the latter equals 
\begin{multline}\label{E:kk1}
(f(x) - f(y))\cdot\frac 1 {2\pi i}\,\oint_{|z| = 1}\; K_{P_z}(x,y) K_{Q_z}(y,x)
\,\frac{dz} z \\
- \frac 1 {2\pi i}\,\oint_{|z| = 1}\;\frac {\p}{\p z}(K_{P_z}(x,y))\,K_{Q_z}
(y,x)\,dz.
\end{multline}

\smallskip

Because of \eqref{E:star2}, to complete the calculation of $\btr [P,Q]$ all we need to do is apply the (matrix) trace to \eqref{E:kk1} and integrate it over $W_0 \times W_0$. An application of this procedure to the second term of \eqref{E:kk1} results in
\[
- \frac 1 {2\pi i}\;\oint_{|z| = 1} \Tr \left(\frac {\p P_z}{\p z}\cdot Q_z
\right)\,dz.
\]
Regarding the first term, we obtain
\begin{multline}\notag
\iint_{W_0\times W_0} (f(x) - f(y))\,\tr\, (K_{P_z}(x,y) K_{Q_z}(y,x))\,dx\,dy \\
= \int_{W_0} f(x)\,\left(\int_X \tr\,(K_{P_z}(x,y)K_{Q_z} (y,x))\,dy\right) dx \\
- \int_{W_0} f(y)\,\left(\int_X \tr\, (K_{P_z}(x,y) K_{Q_z} (y,x))\,dx\right) dy, 
\end{multline}
which equals
\[
\int_{W_0} f(x)\cdot\tr\, \left(K_{\,P_z Q_z}\,(x,x) - K_{\,Q_z P_z}\,(x,x)\right)\,dx.
\]
 
\begin{theorem}\label{T:comm}
Let $P = \D^-\exp(-s \D^+\D^-)$ and $Q = \D^+\exp(-t \D^-\D^+)$ be operators 
on $\zp$ then 
\begin{multline}\notag
\btr\,[P,Q] = \\ \frac 1 {2\pi i}\;\oint_{|z| = 1} 
\left(\int_{W_0}f(x)\cdot\tr\,\left(K_{\,P_z Q_z}\,(x,x)-K_{\,Q_z P_z}\,(x,x)\right)
\,dx\,\right) \frac{dz} z \\
- \frac 1 {2\pi i}\;\oint_{|z| = 1} \Tr \left(\frac {\p P_z}{\p z}\cdot Q_z
\right)\,dz.
\end{multline}
\end{theorem}

\smallskip

\begin{proof}
The proof is contained in the lengthy calculation that precedes the statement of this theorem; here is a summary. We begin by using some elementary calculus of smoothing kernels on $\zp$ to derive formula \eqref{E:star}. Lemmas \ref{L:4.3} and \ref{L:delta+} show how the right hand side of that formula can be expressed in terms of smoothing kernels on $\tilde X$. Passing to the limit, we arrive at formula \eqref{E:star2}. In Section \ref{S:step2}, that formula is converted into the formula claimed in the theorem using Parseval's relation for the Fourier--Laplace transform. 
\end{proof}

%%%%%%%%%%%%%%%%%%%%%%%%%%%%%%%%%%%%%%%%%%%%%%%

\section{The end-periodic index theorem}\label{S:proof}
Let $\zp$ be an even dimensional end-periodic manifold whose end is modeled 
on the infinite cyclic cover $\tilde X$ of $X$, and assume that the chiral 
Dirac operators
$\D^{\pm} = \D^{\pm} (\zp):\; L^2_1 (\zp,\Sm^{\pm}) \to L^2 (\zp,\Sm^{\mp})$
are Fredholm. For the sake of brevity, introduce the notation
\[
\sbtr (e^{-t \D^2}) = \btr (\exp(-t \D^-\D^+)) - \btr (\exp (-t \D^+\D^-)).
\]
Our calculation of the index of $\D^+$ will rely on the following two 
formulas. 

The first formula is obtained by straightforward differentiation using 
the identity $\D^-\exp(-t \D^+\D^-) = \exp(-t \D^-\D^+) \D^-$\,:
\[
\frac d {dt}\, \sbtr (e^{-t \D^2}) = - \btr\, [\D^-,\D^+\exp(-t \D^-\D^+)].
\]
The second formula is the formula of Theorem \ref{T:comm} with $P = \D^-$ 
and $Q = \D^+\exp(-t\D^-\D^+)$. Since $P = \D^-$ is not a smoothing operator, 
this needs a little justification.

\begin{lemma}
The formula of Theorem \ref{T:comm} holds as stated for $P = \D^-$ and 
$Q = \D^+\exp(-t\D^-\D^+)$.
\end{lemma}

\begin{proof} 
Consider the family $P_s = \D^- \exp(-s \D^+\D^-)$ of smoothing operators and apply Theorem \ref{T:comm} to $P_s$ and $Q$ to derive a formula for $\btr [P_s,Q]$. To obtain the formula for $\btr [P,Q]$, simply pass to the limit in this formula as $s \to 0$. The equality
\[
\lim_{s\to 0}\; \btr\, [P_s,Q]\; =\; \btr [P,Q]
\] 
follows from Lemma \ref{L:btr}: for any positive $t$, the regularized trace of $P_s Q = \D^- \D^+ \exp(-(s+t) \D^- \D^+)$ is a continuous function of $s \ge 0$, and so is the regularized trace of $QP_s$.
\end{proof}

Together, the two aforementioned formulas result in 
\begin{multline}\notag
\frac d {dt}\, \sbtr (e^{-t \D^2}) = \\ 
- \frac 1 {2\pi i}\;\oint_{|z|=1}
\int_{W_0} f\cdot\tr \left(K_{\,\D^-_z \D^+_z\exp(-t \D^-_z \D^+_z)} - 
K_{\,\D^+_z \D^-_z\exp(-t \D^+_z \D^-_z)}\right)\,dx\,\frac{dz} z \\
- \frac 1 {2\pi i}\;\oint_{|z| = 1}\;
\Tr \left(d f\cdot \D^+_z \exp (-t \D^-_z \D^+_z)\right)\,\frac {dz} z
\end{multline}
because $\p \D^-_z/\p z = - df/z$. Since
\[
\frac d {dt}\, \exp(-t \D^{\mp}_z \D^{\pm}_z) = 
- \D^{\mp}_z \D^{\pm}_z\exp(-t \D^{\mp}_z \D^{\pm}_z),
\]
we can write 
\begin{multline}\label{E:comm}
\frac d {dt}\, \sbtr (e^{-t \D^2}) = \\ 
\frac 1 {2\pi i}\cdot\frac d {dt}\;\oint_{|z|=1}\left(
\int_{W_0} f\cdot\tr \left(K_{\,\exp(-t \D^-_z \D^+_z)} - 
K_{\,\exp(-t \D^+_z \D^-_z)}\right)\,dx\right)\,\frac{dz} z \\
- \frac 1 {2\pi i}\;\oint_{|z| = 1}\;
\Tr \left(d f\cdot \D^+_z \exp (-t \D^-_z \D^+_z)\right)\,\frac {dz} z
\end{multline}
Integrating the latter formula with respect to $t\in (0,\infty)$, we obtain
an identity whose individual terms are described one at a time in the three
subsections that follow. 

%%%%%%%%%%%%%%%%%%%%%%%%%%%%%%%%%%%%%%%%%%%%%%%

\subsection{The left-hand side}
Integrate the left hand side of \eqref{E:comm} with respect to $t \in (0,\infty)$ to obtain
\[
\lim_{t\to\infty}\,\sbtr (e^{-t \D^2})\;-\;\lim_{t\to 0+}\,\sbtr (e^{-t \D^2}).
\]

Let us first address the limit as $t \to 0$. According to Roe \cite{roe1}, on any manifold of bounded geometry (of even dimension $n$), of which end-periodic manifolds are a special case, we have an asymptotic expansion 
\[
\tr \left(K_{\exp(-t\D^-\D^+)}(x,x)\right)\; \sim\; t^{-n/2}\,\sum_{k \ge 0}\; 
t^k\cdot \psi_k (x),
\]
whose remainder terms, implicit in the asymptotic expansion, are uniformly 
bounded in $x$. The same local calculation as in \cite[page 146]{BGV} then 
gives an asymptotic expansion
\[
\tr\left(K_{\exp(-t\D^-\D^+)}(x,x) - K_{\exp (-t \D^+ \D^-)}(x,x)\right)\;\sim\;
t^{-n/2}\,\sum_{k \ge n/2}\; t^k\cdot \alpha_k (x),
\]
where $\alpha_k (x)$ is locally computable in terms of curvatures and 
their derivatives, and $\alpha_0(x)$ is the index form. In particular, 
on the end-periodic manifold $\zp$ we have
\begin{equation}\label{E:ahat}
\lim_{t\to 0+}\;\tr\left(K_{\exp(-t\D^-\D^+)}(x,x) - K_{\exp (-t \D^+ \D^-)}(x,x)
\right) = \ii(\D^+(\zp))(x)
\end{equation}
uniformly in $x \in \zp$, and similarly on $\tilde X$.

\begin{proposition}\label{P:zero}
At the level of regularized traces, we have 
\[
\lim_{t\to 0+}\, \sbtr (e^{-t \D^2})\; =\; \int_Z \ii(\D^+(Z)).
\]
\end{proposition}

\begin{proof}
Use formula \eqref{E:btr} to write 
\[
\sbtr (e^{-t \D^2}) = \lim_{N \to \infty} s_N (t)
\]
with the functions $s_N: (0,\infty) \to \R$ defined by
\begin{equation}\label{E:snt}
s_N(t)\;=\;\int_{Z_N} \str(t,x)\,dx - (N+1) \int_{W_0} \widetilde\str(t,x)\,dx,
\end{equation}
where $\str (t,x) = \tr\left(K_{\exp(-t\D^-\D^+)}(x,x) - K_{\exp(-t\D^+ \D^-)}(x,x)\right)$ for $x \in \zp$, and $\widetilde\str(t,x)$ is given by the same formula for $x \in \tilde X$. It follows from \eqref{E:ahat} that, for any fixed $N$,
\[
\lim_{t \to 0+} s_N(t)\, = \int_{Z_N} \ii(\D^+(\zp)) - (N+1) \int_{W_0} 
\ii(\D^+(\tilde X)) = \int_Z \ii(\D^+(Z)).
\]

\smallskip\noindent
On the other hand, it follows as in the proof of Lemma \ref{L:btr} that the 
limit 
\[
\lim_{N\to\infty} s_N (t)\, =\, \sbtr\, (e^{-t \D^2})
\]
is uniform on all bounded intervals. Therefore, the repeated limits 
\[
\lim_{N\to\infty}\; \lim_{t \to 0+}\, s_N(t)\qquad\text{and}\qquad
\lim_{t\to 0+}\; \lim_{N\to\infty}\, s_N(t)
\]
exist and are equal to each other, which justifies the following 
calculation\,:
\begin{multline}\notag
\lim_{t\to 0+} \sbtr(e^{-t \D^2}) = \lim_{t\to 0+}\; \lim_{N\to\infty}\,s_N(t) =
\\ \lim_{N\to\infty}\; \lim_{t \to 0+}\, s_N(t) = \lim_{N\to \infty} \int_Z 
\ii(\D^+(Z)) = \int_Z \ii (\D^+(Z)).
\end{multline}
\end{proof}

Let us now investigate the limit of $\sbtr (e^{-t \D^2})$ as $t \to \infty$.

\begin{proposition}\label{P:long}
At the level of regularized traces, we have 
\[
\lim_{t\to \infty}\, \sbtr (e^{-t \D^2})\; =\; \ind \D^+ (\zp).
\]
\end{proposition}

\begin{proof} 
We will only show that $\lim \btr (e^{-t \D^-\D^+}) = \dim\ker \D^+ (\zp)$ 
since the proof of the statement with the roles of $\D^-$ and $\D^+$ 
reversed is identical. Let $K_0(t;x,y) = K(t;x,y) - K_{P_+} (t;x,y)$ as 
in Section \ref{S:long-time}, where $P_+$ is the projector onto $\ker 
\D^+(\zp)$. Write
\begin{multline}\notag
\lim_{t \to \infty} \left(\btr (e^{-t\D^-\D^+}) - \dim \ker \D^+(\zp)\right) \\
= \lim_{t \to \infty} \lim_{N \to \infty} \left(\int_{Z_N} \tr (K_0(t;x,x))\,dx - 
(N+1) \int_{W_0} \tr (\tilde K(t;x,x))\,dx\right) \\
= \lim_{t \to \infty}\; \left(\;\sum_{k = 0}^{\infty}\; \int_{W_0} 
\tr\,(K_0 (t;x+k,x+k) - \tilde K(t;x+k,x+k))\,dx\right) \\
+ \lim_{t \to \infty}\; \int_Z \tr (K_0 (t;x,x))\,dx
\end{multline}

\medskip\noindent
as a sum of two limits, by breaking $Z_N$ into $Z$ and $N+1$ copies of $W_0$. It is immediate from Proposition \ref{P:mu} that the latter limit vanishes. As for the former limit, we have the following two estimates\,:
\begin{multline}\notag
|K_0 (t;x+k,x+k) - \tilde K (t;x+k,x+k)| \\
\le\; |K (t;x+k,x+k) - \tilde K (t;x+k,x+k)| + |K_{P_+} (t;x+k,x+k)| \\
\le\; C_1\, e^{\,\alpha t}\,e^{-\gamma\,k/t} + C_2\,e^{-\delta k}
\end{multline}
by Proposition \ref{P:diff} and Proposition \ref{P:KP+}, and (assuming without
loss of generality that $t \ge 1$)
\begin{multline}\notag
|K_0 (t;x+k,x+k) - \tilde K (t;x+k,x+k)| \\
\le \; |K_0 (t;x+k,x+k)| + |\tilde K (t;x+k,x+k)|\; \le\; C e^{-\mu t} 
\end{multline}
by Proposition \ref{P:mu} and Proposition \ref{P:Ktilde}. We will use the 
latter estimate for the terms in the series
\[
\sum _{k=0}^{\infty}\; |K_0 (t;x+k,x+k) - \tilde K (t;x+k,x+k)|
\]
with $k \le (\alpha + \mu)\,t^2/\gamma$, and the former for the terms 
with $k > (\alpha + \mu)\,t^2/\gamma$. The series is then bounded from
above by

\[
\sum_{k \le (\alpha + \mu)t^2/\gamma}\,C\,e^{-\mu t}\;\le\;
C\,(\alpha+\mu)\,t^2\, e^{-\mu t}/\gamma
\]
plus
\begin{multline}\notag
\sum_{k > (\alpha + \mu)\,t^2/\gamma} C_1\,e^{\,\alpha t}e^{-\gamma k/t}\;\le\;
C_1\,\sum_{\ell=0}^{\infty}\;e^{\,\alpha t -\gamma((\alpha +\mu)\,t^2/\gamma + \ell)/t} \\
\;\le\;C_1\,e^{-\mu t}\,\sum_{\ell=0}^{\infty}\;e^{-\gamma\ell/t}\;=\; 
\frac {C_1\,e^{-\mu t}}{1 - e^{-\gamma/t}}
\end{multline}
and plus
\[
\sum_{k > (\alpha + \mu)\,t^2/\gamma} C_2\,e^{-\delta k}\;\le\;
C_2\,\sum_{\ell=0}^{\infty}\;e^{-\delta((\alpha+\mu)\,t^2/\gamma + \ell)}\;=\;
\frac {C_2\,e^{-\delta(\alpha+\mu)\,t^2/\gamma}}{1 - e^{-\delta}}. 
\]

\bigskip\noindent
It is an easy calculus exercise to show that all of the above three terms limit to zero as $t \to \infty$, which completes the proof. 
\end{proof}

\medskip

%%%%%%%%%%%%%%%%%%%%%%%%%%%%%%%%%%%%%%%%%%%%%%%%

\subsection{The first term on the right}
Integrating the first term on the right hand side of \eqref{E:comm} with respect to $t \in (0,\infty)$, we obtain
\medskip
\begin{gather}
\lim_{t\to\infty}\;\frac 1 {2\pi i}\;\oint_{|z|=1}\left(
\int_{W_0} f\cdot\tr \left(K_{\,\exp(-t \D^-_z \D^+_z)} - 
K_{\,\exp(-t \D^+_z \D^-_z)}\right)\,dx\right)\,\frac{dz} z \notag \\
- \lim_{t\to 0+}\;\frac 1 {2\pi i}\;\oint_{|z|=1}\left(
\int_{W_0} f\cdot\tr \left(K_{\,\exp(-t \D^-_z \D^+_z)} - 
K_{\,\exp(-t \D^+_z \D^-_z)}\right)\,dx\right)\,\frac{dz} z. \notag
\end{gather}

\medskip
As $t \to \infty$, each of the operators $\exp (-t \D^-_z \D^+_z)$ on $X$ converges to the orthogonal projection onto $\ker \D_z^+$. This projection is zero for all $z$ on the unit circle $|z| = 1$ because $\ker \D_z^+ = 0$ for all such $z$; see Proposition \ref{P:taubes}. Since $X$ is closed, we have the uniform convergence of smoothing kernels, $K_{\exp (-t \D^-_z \D^+_z)} (x,x) \to 0$; see for instance \cite[Lemma 1.2]{roe2}. Similarly, we have a uniform limit $K_{\exp (-t \D^+_z \D^-_z)} (x,x) \to 0$. This implies that the first limit in the above formula vanishes. 

Concerning the second limit, note that the operators $\D^{\pm}_z$ are Dirac operators twisted by the connection $-\ln z\;df$ in a complex line bundle $E_z$. Since we are assuming that $|z| = 1$, this is a unitary connection hence we have (on the closed manifold $X$)
\[
\lim_{t\to 0+}\,\tr \left(K_{\,\exp(-t \D^-_z \D^+_z)} - K_{\,\exp(-t \D^+_z \D^-_z)}
\right) = \ii(\D^+(X)) \ch (E_z),
\]
which in fact equals simply $\ii(\D^+(X))$ because the line bundles $E_z$ 
are flat. In particular, the integrand in the second limit is independent 
of $z$ so the $z$--integration results in 
\begin{equation}\label{E:fA}
\int_{W_0} f\cdot \ii(\D^+(X)).
\end{equation}
Since
\[
\int_X \ii(\D^+(X)) = \ind \D^+ (X) = 0
\]
by Corollary \ref{C:fred}, the form $\ii(\D^+(X))$ is exact. Choose a differential form $\omega$ on $X$ such that 
\[
d\omega = \ii(\D^+(X)).
\]
Recall that $f: W_0 \to \R$ is a function on $W_0$ but not on $X$. Denote the two boundary components of $W_0$ by $\p_- W_0 = Y_0$ and $\p_+ W_0 = Y_1$ (of course, $Y_0 = -Y$ and $Y_1 = Y$) and observe that $f|_{Y_1} = f|_{Y_0} + 1$. Apply Stokes' Theorem to \eqref{E:fA} to obtain 
\begin{multline}\notag
\int_{W_0} f\cdot \ii(\D^+(X)) = \int_{W_0} f\cdot d\omega
= \int_{Y_1} f\omega - \int_{Y_0} f\omega - \int_{W_0} df\wedge\omega \\
= \int_Y \left(f|_{Y_1}- f|_{Y_0}\right)\omega - \int_X df\wedge\omega =
\int_Y\omega - \int_X df\wedge\omega.
\end{multline}

\begin{remark}\label{R:product}
Suppose that $Y$ has a product neighborhood in $X$ with a product metric,
and that $df$ is supported in that neighborhood. Then one can easily check 
that 
\[
\int_Y\omega - \int_X df\wedge\omega = 0.
\]
\end{remark}

\smallskip

%%%%%%%%%%%%%%%%%%%%%%%%%%%%%%%%%%%%%%%%%%%%%%%%

\subsection{The second term on the right}
Integrate the second term on the right hand side of \eqref{E:comm} with respect to $t \in (0,\infty)$ to obtain
\smallskip
\begin{equation}\label{E:eta}
- \frac 1 {2\pi i}\,\int_0^{\infty}\oint_{|z| = 1}\;
\Tr \left(df\cdot \D^+_z \exp (-t \D^-_z \D^+_z)\right)\,\frac {dz} z\,dt,
\end{equation}

\medskip\noindent
which equals negative one half times the $\eta$--invariant $\eta(X)$ defined 
in \eqref{E:etaintro}. This completes the proof of Theorem \ref{T:aps}. 

\bigskip

%%%%%%%%%%%%%%%%%%%%%%%%%%%%%%%%%%%%%%%%%%%%%%%

\section{The periodic $\eta$--invariant}\label{S:end-eta}
In this section, we will try to get a clearer idea of what the periodic $\eta$--invariant \eqref{E:etaintro} represents, and how it relates to the classical $\eta$--invariant of Atiyah, Patodi and Singer. 

%%%%%%%%%%%%%%%%%%%%%%%%%%%%%%%%%%%%%%%%%%%%%%%%

\subsection{A spectral interpretation}
The classical $\eta$--invariant is a spectral invariant; we will obtain a similar, if not as explicit, formula for the periodic $\eta$--invariant. We will continue to assume that the $L^2$--closure of $\D^+(\zp)$ is Fredholm or, equivalently, that the operators $\D_z^{\pm} = \D^{\pm} (X) - \ln z\cdot df$ are invertible when $|z| = 1$. 

Theorem 4.6 in \cite{MRS} states that the family $\D^+_z$ is meromorphic in the variable $z \in \mathbb C^*$, as is the family $\D^-_z$. The poles $z_k$ of the family $\D^+_z$, which match the poles of the family $\D^-_z$, form the spectral set of $\D^{\pm}_z$; see Section \ref{S:fred}. Note that the analysis in \cite{MRS} was only done for spin Dirac operators in dimension four but it readily extends to the situation at hand.

We wish to relate \eqref{E:etaintro} to the spectral set of $\D_z^{\pm}$. To this end, observe that 
\[
- \frac d {dt}\;\left(df\cdot (\D_z^-)^{-1} e^{-t \D_z^- \D_z^+}\right)\; =\; 
df\cdot \D_z^+\, e^{-t \D_z^- \D_z^+}.
\]
Since $X$ is compact, the kernel of $df\cdot (\D_z^-)^{-1}\,e^{-t \D_z^-\D_z^+}$ converges uniformly to zero as $t \to \infty$ as long as $|z| = 1$, hence we can write
\begin{multline}\notag
\eta(X)\;=\;\frac 1 {\pi i}\,\int_0^{\infty}\;\oint_{|z|=1}\;\Tr(df\cdot \D_z^+ 
e^{-t \D_z^- \D_z^+})\,\frac {dz}z\,dt \\ =\; 
\lim_{t \to 0}\;\frac 1 {\pi i}\,\oint_{|z|=1}\;\Tr(df\cdot (\D_z^-)^{-1} 
e^{-t \D_z^- \D_z^+})\,\frac{dz} z .
\end{multline}

\smallskip\noindent
As we explain next, after passing to the limit as $t \to 0$ \emph{under} the integral, one can make the right hand side of this formula into a series with summation over the spectral set of $\D^{\pm}_z$. The invariant $\eta(X)$ can then be viewed as a regularization of this (divergent) series. 

A direct calculation shows that, for any $z \in \C^*$ away from the spectral set of $\D^{\pm}_z$, we have 
\begin{equation}\label{E:K}
df \cdot (\D^-_z)^{-1}\; =\; K\cdot (I - \ln z\cdot K)^{-1},\quad\text{where $K = df\cdot \D^- (X)^{-1}$}.
\end{equation}
The operator $K: L^2(X;\Sm^+) \to L^2(X;\Sm^+)$ is compact because both operators $df: L^2 \to L^2$ and $\D^-(X)^{-1}: L^2 \to L^2_1$ are bounded so that their composition factors through the compact embedding $L^2_1 \to L^2$. As a compact operator, $K$ has a discrete spectrum and a basis of generalized eigenspinors. 

Write $\ln z_k = \mu_k + 2\pi in$, $n \in \Z$, for a choice of branch of $\ln z$, then the eigenvalues of $K$ are of the form $\lambda_{k,n} = 1/(\mu_k + 2\pi in)$. It follows from \eqref{E:K} that, for every $z = e^{is}$, $s \in \mathbb R$, the operator $df \cdot (\D^-_z)^{-1}$ restricted to the generalized eigenspace of $K$ corresponding to $\lambda_{k,n}$ is isomorphic to a sum of Jordan cells with $1/(\mu_k + 2\pi in - is)$ along the diagonal. Also note that, for any given $k$, the generalized eigenspaces of $K$ corresponding to $\lambda_{k,n}$ for different $n$ are isomorphic to each other. These (finite dimensional) spaces will be denoted by $V_k$.

 Since $L^2 (X;\Sm^+)$ is a sum of the generalized eigenspaces of $K$, one can formally write 
\medskip
\[
\Tr\left(df\cdot (\D_z^-)^{-1}\right)\,=\,\sum_k\;\sum_n\;\frac 1 {\mu_k - i\,(s - 2\pi n)}\;\, \dim V_k.
\]

\medskip\noindent
Integrate the right hand side of this formula with respect to $s \in [0,2\pi]$ and convert the summation over $n \in \mathbb Z$ into an improper integral to obtain
\medskip
\[
\sum_k\,\left(\;\frac 1 {\pi}\;\int_{-\infty}^{\infty}\;\frac 1 {\mu_k - is}\;ds\right)\,\dim V_k\,.
\]

\smallskip\noindent
The summation here extends over the points $z_k$ in the spectral set of the family $\D_z^{\pm}$. Since
\[
\frac 1 {\pi}\;\int_{-\infty}^{\infty}\;\frac {ds} {\mu_k - is}\; =\;\sign \left(\Re\,\mu_k\right),
\]

\medskip\noindent
we arrive at the promised interpretation of $\eta(X)$ as a regularization of the divergent series
\smallskip
\begin{equation}\label{E:reg}
\sum_k\;\sign\,\ln\, |z_k|\,\cdot\,\dim V_k.
\end{equation}

This series can be viewed as a `spectral asymmetry' of the family $\D_z^+$ with respect to the unit circle $|z| = 1$, that is, informally, the number of spectral points with $|z| > 1$ minus the number of spectral points with $|z| < 1$. 

\begin{remark}
We used the notation $d(z_k)$ in \cite[Section 6.3]{MRS} for the dimension of the solution space of the system \eqref{E:dz} responsible for the asymptotic behavior of the kernel of the spin Dirac operator $\D^+(\zp)$ over the end. It is a straightforward linear algebra exercise to show that $d(z_k) = \dim V_k$.
\end{remark}

%%%%%%%%%%%%%%%%%%%%%%%%%%%%%%%%%%%%%%%%%%%%%%%

\subsection{$\eta$-invariant and von Neumann trace}
Here is another interpretation of the periodic $\eta$--invariant using the trace in the von Neumann algebra of bounded $L^2$--operators on $\tilde X$.

\begin{proposition}
Let $\D^+$ and $\D^-$ be periodic Dirac operators on $\tilde X$, and $\tau$ 
the von Neumann trace on $\tilde X$, then 
\[
\eta(X)\; =\; 2\int_0^{\infty} \tau \left(df\cdot\D^+ \exp(-t\D^-\D^+)\right)\,dt.
\]
\end{proposition}

\medskip
We use the following definition of von Neumann trace; see Atiyah \cite{atiyah}. Let $\tilde K (x,y)$ be the smoothing kernel of the operator $df\cdot\D^+ \exp(-t\D^-\D^+)$ on $\tilde X$. Then $\tr\,(\tilde K(x,x))$ is a periodic function; see \eqref{E:per}, hence it is lifted from a function on $X$; integrate the latter function over $X$ to get $\tau(df\cdot \D^+ \exp (-t \D^-\D^+)$.

\begin{proof}
Integrate the equation \eqref{E:kkz} with $x = y\,\in W_0$ to obtain
\[
\int_{W_0} \tilde K(x,x)\,dx\; =\; \frac 1 {2\pi i}\; \oint_{|z| = 1} \left(\int_X
K_z (x,x)\;dx\right)\,\frac {dz} z.
\]
If $\tilde K$ is the smoothing kernel of $df\cdot \D^+ \exp(-t \D^-\D^+)$ on 
$\tilde X$ then $K_z$ is the smoothing kernel of the operator $df \cdot 
\D_z^+ \exp(-t \D_z^-\D_z^+)$. Applying the matrix trace we obtain
\[
\int_{W_0} \tr\tilde K(x,x)\,dx\;=\;\frac 1 {2\pi i}\;\oint_{|z| = 1} \Tr (df
\cdot \D^+_z \exp(-t \D^-_z\D^+_z))\, \frac {dz} z.
\]
Integration with respect to $t$ completes the proof.
\end{proof}

%%%%%%%%%%%%%%%%%%%%%%%%%%%%%%%%%%%%%%%%%%%%%%%%

\subsection{The product case}\label{S:product}
Let $X = S^1 \times Y$ with product metric and orientation, and choose $f = \theta$. Then $\D^+(X) = d\theta\cdot(\p/\p \theta - \D(Y))$ hence, if we write $z = e^{is}$,\, $0 \le s \le 2\pi$, we have
\[
\D_z^+ = d\theta\cdot(\p/\p\theta - \D - is)\quad\text{and}\quad
\D_z^- = (\p/\p\theta + \D - is)\,d\theta,
\]
where $\D = \D(Y)$ is the self-adjoint Dirac operator on $Y$. The spectral points of $\D^{\pm}_z$ are easy to compute: they are of the form $z = e^{\lambda}$, where $\lambda$ is an eigenvalue of $\D$. To calculate the trace $\Tr \left(df\cdot \D^+_z\,e^{-t \D_z^-\D_z^+}\right)$, we will take advantage of the basis of eigenspinors $\psi_{n,\lambda} = e^{2\pi i n\theta}\varphi_{\lambda}$, where $n$ is an arbitrary integer and $\D\varphi_{\lambda} = \lambda \varphi_{\lambda}$ (note that $\lambda \ne 0$ since we assume $\ker \D = 0$). A direct calculation in this basis gives
\[
\Tr(df\cdot \D_z^+ e^{-t \D_z^- \D_z^+}) = 
\sum_{n,\lambda}\;(\lambda + i (s - 2\pi n))\;e^{-t(\lambda^2 + (s - 2\pi n)^2)}
\]
and 
\begin{multline}\notag
\frac 1{\pi i}\,\oint_{|z|=1}\Tr(df\cdot\D_z^+ e^{-t \D_z^-\D_z^+})\,\frac {dz} z \\
= \frac 1 {\pi}\,\int_0^{2\pi}\;\Tr(df\cdot \D_z^+ e^{-t \D_z^- \D_z^+})\,ds \\ =
\frac 1 {\pi}\,\sum_{\lambda}\;\int_{-\infty}^{\infty} (\lambda + is)\;
e^{-t(\lambda^2 + s^2)}\,ds,
\end{multline}
where we incorporated the summation over $n$ into the improper integral. Next,
use the fact that $s\,e^{-t(\lambda^2 + s^2)}$ is an odd function in $s$ to obtain, after some basic integration,
\smallskip
\[
\frac 1 {\pi i}\,\oint_{|z|=1} \Tr(df\cdot \D_z^+ e^{-t \D_z^- \D_z^+})\,
\frac {dz} z \;=\; \frac 1 {\sqrt{\,\pi t}}\,\sum_{\lambda}\;\lambda\,
e^{-t \lambda^2}.
\]

\smallskip\noindent
Integration with respect to $t$ results in
\smallskip
\[
\eta(X) = \frac 1 {\sqrt{\,\pi}}\;\int_0^{\infty}\;t^{-1/2}\,\left(\sum_{\lambda}\;
\lambda\,e^{-t\lambda^2}\right)\,dt,
\]

\smallskip\noindent
which gives a well-known formula for the $\eta$--invariant of $\D$,
\smallskip
\[
\eta_{\,\D} (Y) = \frac 1 {\sqrt{\,\pi}}\;\int_0^{\infty}\;t^{-1/2}\,
\Tr \left(\D\,e^{-t\D^2}\right)\,dt;
\]

\smallskip\noindent
see for instance \cite{melrose:aps}. Therefore, $\eta (X) = \eta_{\,\D}(Y)$. That this matches the $\eta$-invariant of \eqref{E:apsindex} can be proved as in \cite{aps:I} using the Mellin transform.

%%%%%%%%%%%%%%%%%%%%%%%%%%%%%%%%%%%%%%%%%%%%%%%%

\subsection{Non-product examples}\label{S:inoue}
In the non-product case, direct computations of spectral sets and periodic $\eta$--invariants are very difficult. In this section, we will obtain partial information about the spectral set of Dirac operators for a class of examples called Inoue surfaces. These are compact complex surfaces which belong to type $\rm{VII}_{\,0}$ in Kodaira's classification. They were constructed by Inoue \cite{inoue}, and their most remarkable property is that they do not admit any holomorphic curves. 

We will study some of the simplest Inoue surfaces, those of class $S_M$. These are compact quotients of $\H \times \C$, where $\H = \{\,w = w_1 + i w_2\in \C\;|\;w_2 > 0\,\}$ is the upper half-plane. To construct $X$, start with a matrix $M = (m_{ij}) \in SL(3,\Z)$ with one real eigenvalue $\alpha > 1$ (easily seen to be irrational) and two complex conjugate eigenvalues $\beta \neq \bar\beta$. Let $a = (a_1,a_2,a_3)$ be a real eigenvector corresponding to $\alpha$, and $b = (b_1,b_2,b_3)$ a complex eigenvector corresponding to $\beta$. Let $G_M$ be the group of complex analytic transformations of $\H \times \C$ generated by
\begin{gather}
g_0 (w,z) = (\alpha w, \beta z), \notag \\
g_i (w,z) = (w + a_i,z + b_i),\quad i = 1, 2, 3. \notag
\end{gather}
The group $G_M$ acts on $\H \times \C$ freely and properly discontinuously so that the quotient $X = (\H \times \C)/G_M$ is a compact complex surface. 

Inoue \cite{inoue} showed that $X$ is smoothly a 3-torus bundle over a circle whose monodromy is given by the matrix $M$, and that $b_1 (X) = 1$ and $b_2 (X) = 0$. Define a function $f: \H \times \C \to \R$ by the formula $f(w,z) = \ln w_2/\ln \alpha$. One can easily see that $df$ is a well defined 1-form on $X$ whose cohomology class generates $H^1 (X;\Z) = \Z$.

The surface $X$ admits no global K{\"a}hler metric. We will however consider the following Hermitian metric on $\H \times \C$, called the Tricerri metric~\cite{tricerri:lck,dragomir-ornea:lck},
\[
g\; =\; \frac{dw\otimes d\bar w}{w_2^2}\; +\; w_2\;dz\otimes d\bar z.
\]
For the K{\"a}hler form $\omega$ of this metric on $\H \times \C$, we find $d\omega = d\ln w_2 \wedge \omega$ with the exact torsion form $d\ln w_2 = \ln \alpha\; df$. The metric $g$ is $G_M$--invariant, and so it defines a metric on $X$ which makes $X$ into a locally conformal K{\"a}hler manifold.

As a complex surface, $X$ admits a canonical $\spinc$ structure with respect to which 
\[
\Sm^+ = \Lambda^{0,0} (X) \,\oplus\,\Lambda^{0,2}(X)\quad\text{and}\quad \Sm^- = \Lambda^{0,1} (X).
\]
Let $\D^{\pm}(X)$ be the Dirac operators associated with the Tricerri metric and the canonical $\spinc$ structure on $X$. According to Gauduchon \cite[page 283]{G}, there is an isomorphism 
\begin{equation}\label{E:gaud}
\D^- (X) + \frac 1 4\, \ln\alpha\cdot df\;=\; \sqrt{\,2}\,(\bar\p\,\oplus\,\bar\p^*),
\end{equation}
where
\begin{equation}\label{E:dd}
\bar\p\,\oplus\,\bar\p^*: \;\Omega^{0,1} (X) \longrightarrow \Omega^{0,2}(X)\,\oplus\,\Omega^{0,0}(X)
\end{equation}
is the Dirac--Dolbeault operator on the complex surface $X$. This identity implies that the spectral set of $\D^{\pm}(X)$ is obtained from that of $\bar\p\,\oplus\,\bar\p^*$ via multiplication by $\alpha^{-1/4}$. 

One can check directly that the sections $d\bar w/w_2$ and $d\bar z$ give rise to the spectral points $z = 1$ and $z = \alpha\beta$ of the Dirac--Dolbeault operator, and hence to the spectral points $\alpha^{-1/4}$ and $\alpha^{\,3/4}\beta$ of the operators $\D^{\pm}(X)$. We see in particular that, unlike in the product case, spectral points need not be real. 

The two spectral points we found above lie on the boundary of the annulus $\alpha^{-1/4} \le |z| \le \alpha^{1/4}$. We can prove that there are no other spectral points inside this annulus via a much more involved argument, which uses Fourier analysis on $X$ viewed as a $3$-torus bundle over a circle to reduce the calculation to solving a Sturm--Liouville problem on the real line. We plan to give the details of this argument in a subsequent paper.

The infinite cyclic cover of $X$ is a product $\R \times T^3$ topologically even though it is not metrically. Manifolds with periodic ends that are not products even topologically have recently appeared in our paper~\cite{mrowka-ruberman-saveliev:derham}, which studied the index of the de Rham complex; the examples there include manifolds whose ends arise from the infinite cyclic covers of 2-knot exteriors in the 4-sphere. 

%%%%%%%%%%%%%%%%%%%%%%%%%%%%%%%%%%%%%%%%%%%%%%%%%

\section{Non-Fredholm case}\label{S:non}
In this section, we will extend Theorem \ref{T:aps} to the case when the $L^2$--closure of the operator $\D^+(\zp)$ is not necessarily Fredholm. This extension will correspond, in the product end case, to the general case of the Atiyah--Patodi--Singer index theorem as stated in \eqref{E:apsindex}, without the assumption that $\ker B = 0$.

%%%%%%%%%%%%%%%%%%%%%%%%%%%%%%%%%%%%%%%%%%%%%%%%%

\subsection{Fredholm theory}
We begin by reviewing some material from our paper  \cite{MRS} regarding the family $\D^+_z (X) = \D^+(X) - \ln z\cdot df$. In \cite{MRS}, we only dealt with spin Dirac operators in dimension four, however, all of the results easily extend to cover the more general case at hand. We will restrict ourselves to stating a few relevant results, and refer to \cite{MRS} for proofs. 

Let us assume that the spectral set of the family $\D^+_z (X)$ is discrete, which will be the case, for example, if the conditions of Theorem \ref{T:fred1} are fulfilled. Since we no longer assume that the $L^2$--closure of the operator $\D^+(\zp)$ is Fredholm, some of the spectral points may land on the unit circle $|z| = 1$. 

For any choice of $\delta \in \R$ that makes the operator \eqref{E:sobolev} Fredholm, denote the index of \eqref{E:sobolev} by $\ind_{\delta} \D^+(\zp)$. Given two choices $\delta < \delta'$, we have the following change of index formula 
\begin{equation}\label{E:change}
\ind_{\,\delta} \D^+(\zp) - \ind_{\,\delta'} \D^+(\zp)\; =\; 
\sum_{e^{\delta}< |z| <\, e^{\delta'}}\; d(z);
\end{equation}
see formula (20) of \cite{MRS}. Here, the integer $d(z)$ is defined as in \cite[Section 6.3]{MRS} to be the dimension of the space of solutions $(\phi_1,\ldots,\phi_m)$ of the system
\begin{equation}\label{E:dz}
\begin{cases}
\; D^+_z (X)\, \phi_1 = df\cdot\phi_2, \\
\; \cdots \\
\; D^+_z (X)\, \phi_{m-1} = df\cdot\phi_m, \\ 
\; D^+_z (X)\, \phi_m = 0.
\end{cases}
\end{equation}
Equivalently, $d(z)$ is the number of linearly independent vectors in the kernel of the operator $\D^+ (\tilde X)$ that have the form
\medskip
\begin{equation}\label{E:asymp}
z^{-f(x)}\;\sum_{p=1}^m\;(-1)^{p-1} f(x)^{p-1}\,\phi_p (x)/(p-1)!
\end{equation}

\smallskip

%%%%%%%%%%%%%%%%%%%%%%%%%%%%%%%%%%%%%%%%%%%%%%%%

\subsection{Statement of the theorem}\label{S:thm-stat}
We will say that $\ep \in \mathbb R$ is \emph{small} if the only complex numbers $z$ in the annulus $e^{-2|\ep|} < |z| < e^{2|\ep|}$ for which $\D^+_z (X)$ is non-invertible are those with $|z| = 1$. Formula \eqref{E:change} with $\delta = \ep$ then implies that the index $\ind_{\ep}\D^+(\zp)$ is independent of $\ep$ as long as $\ep$ is small and stays on the same side of zero. We denote this index by $\ind_{\pm} \D^+(\zp)$ according to whether $\ep$ is positive or negative.

Our extension of Theorem \ref{T:aps} will give a formula for $\ind_+\D^+(\zp)$. In order to state it, we need to introduce two new quantities. First, let
\begin{equation}\label{E:h}
h\;=\;\sum_{|z|=1}\;d(z),
\end{equation}
where $d(z)$ are defined by \eqref{E:dz}. Equivalently, $h$ is the term on the right hand side of the change of index formula \eqref{E:change} with small negative $\delta$ and small positive $\delta'$. This second definition implies, in particular, that $h$ is independent of the choice of $f$.

The integer \eqref{E:h} will play the role of $h_B = \dim\ker B$ of the Atiyah--Patodi--Singer theorem \eqref{E:apsindex} in the product case. To be precise, if $X = S^1 \times Y$ with product metric and $f = \theta$ then $\D^+ (X) = d\theta \cdot (\p/\p\theta - \D(Y))$; see Section \ref{S:product}. A straightforward calculation using Fourier analysis then shows that $d(z) = 0$ for all $z \neq 1$ on the unit circle $|z| = 1$, and $h = d(1) = \dim \ker \D(Y)$. Second, let 
\medskip
\begin{equation}\label{E:eta-ep}
\eta_{\ep}(X) = \frac 1 {\pi i}\,\int_0^{\infty}\oint_{|z| = e^{\ep}}\;
\Tr \left(df\cdot \D^+_z \exp (-t (\D^+_z)^* \D^+_z)\right)\,\frac {dz} z\,dt,
\end{equation}

\medskip\noindent
where the integral is understood in a regularized sense: for small positive $t$, the integral 
\medskip
\[
\frac 1 {\pi i}\,\int_t^{\infty}\oint_{|z| = e^{\ep}}\;
\Tr \left(df\cdot \D^+_z \exp (-t (\D^+_z)^* \D^+_z)\right)\,\frac {dz} z\,dt,
\]

\medskip\noindent
has an asymptotic expansion in powers of $t$, and we let $\eta_{\ep}(X)$ equal the constant term in this expansion. Define
\smallskip
\begin{equation}\label{E:etapm}
\eta_{\pm} (X)\;=\; \lim_{\ep \to 0\pm}\,\eta_{\ep}(X)\quad\text{and}\quad 
\eta(X)\;=\;\frac 1 2\;(\eta_+(X) + \eta_-(X)).
\end{equation}

\smallskip\noindent
Since $(\D^+_z)^* = \D_z^-$ on the unit circle $|z| = 1$ this definition of $\eta\,(X)$ matches that in the Fredholm case. Similarly to \eqref{E:reg}, one can interpret $\eta\,(X)$ defined by \eqref{E:etapm} as a regularization of the series 
\[
\sum_{|z_k| \ne 1}\;\sign\, \ln |z_k|\cdot d(z_k).
\]
The equality $\eta(X) = \eta_{\D}(Y)$ proved in Section \ref{S:product} for product manifolds $X = S^1 \times Y$ continues to hold in the non-Fredholm case.

\begin{thm}\label{T:non}
Let $\D^+(\zp)$ be such that the spectral set of $\D^+_z (X)$ is a discrete subset of $\C^*$, and let $\omega$ be a form on $X$ such that $d\omega = \ii(\D^+(X))$. Then 
\[
\ind_+\D^+(\zp)\;=\;\int_Z\;\ii(\D^+(Z)) - \int_Y\;\omega\, + 
\int_X\;df\wedge \omega\, -\, \frac {h + \eta\,(X)}{2}.
\]
\end{thm}

\smallskip

%%%%%%%%%%%%%%%%%%%%%%%%%%%%%%%%%%%%%%%%%%%%%%%%

\subsection{Sketch of the proof}
Given $\ep\in \R$, consider the operators $\D_{\ep} = e^{\ep f}\,\D\,e^{-\ep f} = \D - \ep\,df$ on each of the manifolds $\zp$, $\tilde X$ and $X$, where $f$ stands for both the function $f: \tilde X \to \mathbb R$ and its extension to $\zp$. If $\ep$ is small, the operator $\D^+_{\ep}$ is Fredholm on $\zp$, and its index can be computed mainly as before. The few changes that arise are due to the fact that the full Dirac operator $\D_{\ep}$ is no longer self-adjoint. To be precise, we have $\D_{\ep}^* = \D_{-\ep}$ and $(\D^+_{\ep}\,)^* = \D^-_{-\ep}$ hence $\ind \D^+_{\ep}(\zp) = \dim\ker \D^+_{\ep}(\zp) - \dim\ker \D^-_{-\ep}(\zp)$. 

The proof now goes as follows. Define the regularized trace as in Section \ref{S:b-trace} and introduce the notation
\[
\sbtr(\ep,t)\; =\; \btr(\exp(-t\D^-_{-\ep}\D^+_{\ep})) - 
\btr(\exp(-t\D^+_{\ep}\D^-_{-\ep}))
\]
for operators on $\zp$. Essentially the same argument as in the proof of Proposition \ref{P:long} shows that
\[
\lim_{t\to\infty}\;\sbtr(\ep,t)\; =\;\ind_{\ep} \D^+(\zp).
\]
On the other hand, the limit 
\[
\lim_{t \to 0}\; \sbtr(\ep,t)\;=\;\int_Z\;\ii(\D^+(Z))
\]
of Proposition \ref{P:zero} now needs to be understood in a regularized sense, as the constant term in the asymptotic expansion of $\sbtr(\ep,t)$ in the powers of $t$. It turns into a true limit as $\ep \to 0$. A direct calculation with the easily verified formula $\exp\,(-t\D_{\ep}\D^*_{\ep})\allowbreak \D_{\ep} = \D_{\ep}\,\exp\,(-t\D^*_{\ep}\D_{\ep})$ shows that 
\medskip
\begin{equation}\label{E:ep-comm}
\frac d{dt}\;\sbtr(\ep,t)\;=\;
-\btr [\D^-_{-\ep},\D^+_{\ep}\exp(-t\D^-_{-\ep}\D^+_{\ep})].
\end{equation}

\smallskip\noindent
Repeat the argument of Section \ref{S:trace} to derive a commutator trace formula with $P = \D^-_{-\ep}$ and $Q = \D^+_{\ep}\exp(-t\D^-_{-\ep}\D^+_{\ep})$, and integrate \eqref{E:ep-comm} with respect to $t \in (0,\infty)$, in regularized sense. Passing to the limit as $\ep \to 0$, we arrive as in Section \ref{S:proof} at the formula 
\medskip
\begin{equation}\label{E:prelim}
\ind_{\pm}\D^+(\zp)\;=\;\int_Z\;\ii(\D^+(Z)) - \int_Y\;\omega + \int_X\;df\wedge \omega\,-\,\frac 1 2\,\eta_{\pm} (X).
\end{equation}

\medskip\noindent
Assume now that $\ep > 0$ then, according to the change of index formula \eqref{E:change}, 
\[
\ind_+ \D^+(\zp) - \ind_- \D^+(\zp)\;=\;-h,
\]
hence
\smallskip
\[
\frac 1 2\,\left(\ind_+\D^+(\zp) + \ind_-\D^+(\zp)\right)\;=\;\ind_+\D^+(\zp)\,+\, \frac 1 2\,h.
\]

\smallskip\noindent
Substituting \eqref{E:prelim} into the left hand side of this identity completes the proof. 

%%%%%%%%%%%%%%%%%%%%%%%%%%%%%%%%%%%%%%%%%%%%%%%%

\subsection{Dependence of $\eta(X)$ on orientations}\label{S:orient}
The invariant $\eta(X)$ defined by \eqref{E:etaintro} and, in general, by \eqref{E:etapm}, depends on two choices of orientation: the orientation of $X$ itself, and the sign of the primitive cohomology class $\gamma \in H^1(X;\Z)$ associated with the infinite cyclic cover of $X$. In applications, it is useful to know how changing these two orientations affects $\eta(X)$.

\begin{proposition}\label{P:orient-hom} 
Let $X$ be an oriented compact manifold with a choice of primitive class $\gamma \in H^1(X; \Z)$, and $\eta(X)$ the periodic $\eta$--invariant of a Dirac operator $\D^+(X)$. Then negating $\gamma$ changes the sign of $\eta(X)$.
\end{proposition}

\begin{proof}
Changing the sign of $\gamma = [df]$ amounts to replacing $f$ by $-f$; this has the effect of changing the family $\D^\pm_z$ to the family $\D^\pm_{1/z}$. The change of variables $w = 1/z$ in the integral \eqref{E:eta-ep} then changes $\eta_{\pm}(X)$ to $-\eta_{\mp}(X)$, hence the invariant $\eta(X)$ defined by \eqref{E:etapm} changes sign.
\end{proof}

\begin{proposition}\label{P:orient}
Let $X$ be an oriented spin compact manifold and $\eta(X)$ the periodic $\eta$-invariant of a spin Dirac operator $\D^+(X)$. Denote by $-X$ the manifold $X$ with reversed orientation. Then
\[
\eta(X) + \eta(-X) = -2h.
\]
\end{proposition}

\begin{proof}
First assume that our $\eta$--invariant arises from an index problem for a spin Dirac operator on an end-periodic manifold $\zp$. Apply Theorem \ref{T:non} to this spin Dirac operator twice, first on $\zp$ and then on $\zp$ with reversed orientation, to obtain the index formulas
\medskip
\[
\ind \D^+ (\zp) = \int_Z \ii (\D^+(Z))\; - \int_Y\omega + \int_X df\wedge\omega\; - \frac 1 2\,(h + \eta (X))\quad \text{and}
\]
\[
-\ind \D^+ (\zp) = -\int_Z \ii (\D^+(Z))\; + \int_Y\omega - \int_X df\wedge\omega\; - \frac 1 2\,(h + \eta (-X)).
\]

\bigskip\noindent 
 Adding these formulas together, we obtain the desired formula $\eta(X) + \eta(-X) = - 2h$. In general, use the fact that the spin cobordism group in odd dimensions vanishes over the rationals, and apply the above argument to an end-periodic manifold $\zp$ with multiple ends.
\end{proof}

%%%%%%%%%%%%%%%%%%%%%%%%%%%%%%%%%%%%%%%%%%%%%%%

\subsection{Spectral flow}\label{S:flow}
Let $\D^+_t (\zp)$ be a family of Dirac operators parameterized by $t \in [0,1]$ to which Theorem \ref{T:non} applies, leading to the formula 
\smallskip
\[
\ind_+\D^+_t(\zp)\;=\;\int_Z\;\ii(\D^+_t(Z)) - \int_Y\;\omega_t\, + 
\int_X\;df\wedge \omega_t\, -\, \frac {h_t + \eta_t (X)}{2}\,.
\]

\bigskip\noindent
As $t$ varies, the integral terms in this formula vary continuously while $\ind_+\D^+_t (\zp)$ may have integer jumps at the values of $t$ for which the operator $\D^+_t (\zp)$ fails to be Fredholm. Following Atiyah--Patodi--Singer \cite[(7.1)]{aps:III}, separate the function $\xi(t) = (h_t + \eta_t (X))/2$ into its continuous part $g(t)$ and the integer value part $j(t)$,
\[
\xi(t) = g(t) + j(t),\quad j(0) = 0.
\]
In favorable circumstances, $j(t)$ can be interpreted as the net number of the spectral points $z$ of the family $\D^+_t (X) -\ln z\cdot df$ crossing the unit circle $|z| = 1$ as $t$ varies. This `spectral flow across the unit circle' generalizes the spectral flow of \cite[Section 7]{aps:III}, and reduces to it in the product end case. We studied this spectral flow in \cite{MRS} for the spin Dirac operators associated to a family of Riemannian metrics $g_t$ on a spin 4-manifold $X$, and showed that under certain regularity assumptions it coincides with the change in the count of solutions to the Seiberg-Witten equations.

%%%%%%%%%%%%%%%%%%%%%%%%%%%%%%%%%%%%%%%%%%%%%%%

\section{Periodic $\tw$--invariants}\label{S:rho}
An important extension of the classical Atiyah--Patodi--Singer $\eta$--invariant $\eta_B (Y)$ involves a twist of the operator $B$ by a unitary representation $\rep: \pi_1(Y) \to U(k)$, yielding invariants 
\[
\eta_{B_{\rep}}(Y)\quad\text{and}\quad  \xi_{B_{\rep}}(Y)\, =\, \frac12\, (h_{B_{\rep}} + \eta_{B_{\rep}}(Y)).
\]
Comparing the untwisted and twisted versions yields the $\tw$--invariant
\begin{equation}\label{E:tw}
\tw_{\rep} (Y,B)\; =\; \xi_{B_{\rep}}(Y) - k\cdot \xi_B(Y),
\end{equation}
see \cite[Section 3]{aps:II}, which has had many applications in geometry and topology.  (In the literature, $\tw$ is denoted variously by $\rho$, $\eta$, eta, and perhaps other symbols.) 
When $B$ is the odd signature operator on an oriented $(4n-1)$--dimensional manifold $Y$, as defined in \cite[(4.6)]{aps:I}, the invariants $\tw_{\rep}(Y,B)$ are metric-independent, and so give smooth invariants of $Y$. When $B$ is the Dirac operator, only the reduction of $\tw_{\rep} (Y,B)$ modulo integers is metric independent.  

We will use the same procedure to define periodic $\tw$--invariants for chiral Dirac operators on even-dimensional manifolds $X$ equipped with a unitary representation. This section contains the definition and basic properties of these periodic $\tw$--invariants. In the following Section~\ref{S:psc}, they are used to study metrics of positive scalar curvature on $X$.

%%%%%%%%%%%%%%%%%%%%%%%%%%%%%%%%%%%%%%%%%%%%%%%%

\subsection{Definition of periodic $\tw$--invariants}
Let $X$ be a compact even-dimensional Riemannian manifold with a choice of a primitive cohomology class $\gamma \in H^1 (X;\Z)$, and let $\D^+ = \D^+(X)$ be a chiral Dirac operator associated with a Dirac bundle on $X$ with the property that the spectral set of $\D^+_z$ is discrete. Suppose that in addition we have a  representation $\rep: \pi_1 (X) \to U(k)$ such that the twisted Dirac operator $\D^+_{\rep} = \D^+_{\rep}(X)$ has the property that the spectral set of $(\D^+_{\rep})_z = \D^+_{z\rep}$ is discrete. Let
\[
\xi (X,\D^+_{\rep})\;=\;\frac 1 2 \left(\,h(\D^+_{\rep})\, +\, \eta (X,\D^+_{\rep})\,\right)
\]
and define the periodic $\tw$--invariant by
\[
\tw_{\rep}(X, \D^+)\;=\;\xi (X,D^+_{\rep}) - k\cdot \xi (X,\D^+),
\]
with the periodic $\eta$--invariants defined as in \eqref{E:etapm} and the $h$--invariants as in \eqref{E:h}. The notations we used there did not include the dependence on various choices that we have made, as these did not play much of a role in the analysis. We will now need to keep track of the Riemannian metric $g$ on $X$ and the function $f: \xtilde \to \R$ such that $\gamma = [df]$. We will incorporate these into our notations as needed, by writing (for example) $\tw_{\rep}(X, \D^+, \barf, g)$.

In the special case when the operators $\D^+_z$ and  $\D^+_{z\rep}$ are invertible on the unit circle $|z| = 1$, which is equivalent to the $L^2$--closures of the operators $\D^+(\tilde X)$ and $\D^+_{\rep}(\tilde X)$ being Fredholm, one can use formula \eqref{E:etaintro} for the $\eta$--invariants in the above definition instead of \eqref{E:etapm}.

%%%%%%%%%%%%%%%%%%%%%%%%%%%%%%%%%%%%%%%%%%%%%%%%%

\subsection{Two special cases}\label{S:special}
Let $X = S^1 \times Y$ for some odd--dimensional manifold $Y$. Equip $X$ with a product metric $g = d\theta^2 + g^Y$ and let $\barf = \theta$. If $\rep: \pi_1 (X) \to U(k)$ factors through $\alpha: \pi_1 (Y) \to U(k)$, the spectral sets of $\D^+_z$ and $\D^+_{z\rep}$ are automatically discrete, and we have
\begin{equation}\label{E:rho-prod}
\tw_\rep (X, \D^+, \barf, g) = \tw_\rep (Y, \D,g^Y),
\end{equation}
where $\D$ is the self-adjoint Dirac operator on $Y$. This follows from Section \ref{S:product} and Section \ref{S:thm-stat} in which the periodic $\eta$--invariants and $h$--invariants of $X$ were identified with those of $Y$.

Another important special case comprises spin Dirac operators on manifolds with metrics of positive scalar curvature. Since such metrics are of particular interest in both this and next sections, we will record the following simple lemma.

\begin{lemma}\label{L:vanish}
Let $\zp$ be an end-periodic spin manifold with an end-periodic metric $g$ of positive scalar curvature. Then the $L^2$--closure of the associated spin Dirac operator $\D^+(\zp)$ is Fredholm and has zero index. The same is true for the twisted operator $\D^+_{\rep}(\zp)$ associated with any representation $\alpha: \pi_1(\zp) \to U(k)$.
\end{lemma}

\begin{proof}
Since the end-periodic metric $g$ has positive scalar curvature, the operators  $\D^+(\zp)$ and $\D^+_{\rep}(\zp)$ are {\em uniformly invertible at infinity}, which implies that their $L^2$--closures are Fredholm; see Gromov--Lawson \cite{gromov-lawson:complete}.  Alternatively, assume that the end of $\zp$ is modeled on $\tilde X \to X$, and use the Lichnerowicz formula~\cite{lichnerowicz:spinors} on $X$ to prove that $\D^+_z(X)$ and $\D^+_{z\rep}(X)$ are invertible on the unit circle $|z| = 1$. The statement about Fredholmness now follows from Proposition~\ref{P:taubes}.  Applying the Lichnerowicz formula once more, this time on $\zp$, we conclude that the operators $\D^+(\zp)$ and $\D^+_{\rep}(\zp)$ are invertible, hence their indices vanish.
\end{proof}

It follows from the proof of Lemma~\ref{L:vanish} that, whenever $X$ is a spin manifold with a metric $g$ of positive scalar curvature and the associated spin Dirac operator $\D^+$, the operators $\D^+_z$ and $\D^+_{z\alpha}$ are invertible on the unit circle and hence the invariants $\tw_{\rep}(X,\D^+,f,g)$ are well defined for all representations $\alpha: \pi_1 (X) \to U(k)$.

%%%%%%%%%%%%%%%%%%%%%%%%%%%%%%%%%%%%%%%%%%%%%%%%%

\subsection{Dependence on choices}
In this section, we will use Theorem~\ref{T:non} to study how the invariants $\tw_{\rep}(X,\D^+,f,g)$ depend on the choices of $f$ and $g$. 

Choose a submanifold $Y\subset X$ dual to the generator $\gamma \in H^1(X;\Z)$. Then 
\begin{equation}\label{E:xtilde}
\xtilde = (\,\underbrace{\ldots \cup W_{-2} \cup W_{-1}}_{\xm}\,)\, \cup\,W_0\,\cup\, (\,\underbrace{W_1\,\cup\,W_2\,\cup\ldots}_{\xp}\,)
\end{equation}
where $W_k$ are isometric copies of the fundamental segment $W$ obtained by cutting $X$ open along $Y$. We will view $\xtilde$ as the union of $W_0$ with two ends $X_+$ and $X_-$ and use Theorem \ref{T:non} to compute the indices of the operators $\D^+(\tilde X)$ and $\D^+_{\rep}(\tilde X)$.

The statement of Theorem \ref{T:non} makes use of functions $f_{\pm}: \tilde X \to \R$ associated with the ends $X_+$ and $X_-$ such that $\gamma = [df_+]$ and $-\gamma = [df_-] \in H^1 (X;\Z)$, respectively. For example, given a function $\barf: \tilde X \to \R$ with $\gamma = [df]$, one could use $f_+ = f$ and $f_- = -f$ as such functions. In addition, Theorem \ref{T:non} requires a choice of local index forms and transgressed classes. Denote by $\ii(\D^+(X))$ the local index form for $\D^+(X)$ then $\ii(\D^+_{\rep}(X)) = k \cdot \ii(\D^+(X))$. Our assumption that the spectral sets of $\D^+_z(X)$ and $\D^+_{z\rep}(X)$ are discrete ensures that the operators $\D^+(X)$ and $\D^+_{\rep}(X)$ have zero index, and hence the forms $\ii(\D^+(X))$ and $\ii(\D^+_{\rep}(X))$ are exact. If $\omega$ is a transgressed class such that $d\omega = \ii(\D^+(X))$ then evidently $d(k \cdot \omega) = \ii(\D^+_{\rep}(X))$.

\begin{lemma}
The invariant $\tw_{\rep}(X, \D^+, \barf, g) $ does not depend on the choice of the function $\barf$.
\end{lemma}

\begin{proof}
Suppose that $\barf_0$ and $\barf_1: \xtilde \to \R$ are two choices of function $f$ such that $\gamma = [df_0] = [df_1]$. Then we can compute the index of $\D^+(\tilde X)$ in two different ways.  One, as described above, will use $f_- = - f_0$ and $f_+ = f_0$. The other will continue to use $f_- = -f_0$ but will use $f_+ = f_1$.  For the first choice, using Theorem \ref{T:non} and taking advantage of Proposition~\ref{P:orient-hom}, we obtain
\begin{align*}
\ind_+ \D^+ (\tilde X) &= \int_{W_0} \ii (\D^+(X))\;  \\
& - \int_Y\omega + \int_X df_0 \wedge\omega\; - \frac 1 2\,(h + \eta(X, \D^+, \barf_0, g))\\
& + \int_Y\omega - \int_X df_0\wedge\omega\; - \frac 1 2\,(h - \eta(X, \D^+, \barf_0, g)).
\end{align*}
Since $\int_{W_0} \ii (\D^+(X)) = 0$, this implies that $\ind_+ \D^+ (\tilde X) = -h$. Note that the same answer could also be obtained from the change of index formula \eqref{E:change}. Now we use the second choice to obtain
\begin{align*}
\ind_+ \D^+ (\tilde X) &= \int_{W_0} \ii (\D^+(X))\;  \\
& - \int_Y\omega + \int_X df_1 \wedge\omega\; - \frac 1 2\, (h + \eta(X, \D^+, \barf_1, g)) \\
& + \int_Y\omega - \int_X df_0\wedge\omega\;- \frac 1 2\,(h - \eta(X, \D^+, \barf_0, g)),
\end{align*}
from which it follows that 
\begin{equation}\label{E:diffeta}
\xi(X, \D^+, \barf_1, g) - \xi(X, \D^+, \barf_0, g) = \int_X df_1 \wedge\omega\, -  \int_X df_0 \wedge\omega.
\end{equation}

\smallskip\noindent
Here, we used the fact that $h$ is independent of $f$, see the discussion following \eqref{E:h}. Repeating this argument with $\D^+$ replaced by $\D^+_{\rep}$, and taking advantage of the relation between the index forms and transgressed classes for the two operators described above, we obtain
\medskip
\begin{equation}\label{E:diffalpha}
\xi(X, \D^+_{\rep}, \barf_1, g) - \xi(X, \D^+_{\rep}, \barf_0, g) = k \left(\int_X df_1 \wedge\omega\; -  \int_X df_0 \wedge\omega\right).
\end{equation}

\medskip\noindent
Subtracting $k$ times~\eqref{E:diffeta} from~\eqref{E:diffalpha} yields the result.
\end{proof}

From this point on, we remove the function $\barf$ from the notation for the periodic $\tw$--invariant, and turn our attention to the metric dependence. Except in those circumstances where the kernel and cokernel of $\D^+(X)$ have topological interpretations, we do not expect that the $\tw$--invariant is metric independent. Indeed, this is not true even in the product case.

\begin{lemma}\label{L:g-indep}
Let $g_0$ and $g_1$ be Riemannian metrics on $X$ with respect to which the spectral sets of $\D^+_z$ and $\D^+_{z\alpha}$ are discrete. Then
$$
\tw_{\rep}(X, \D^+, g_1)\;=\;  \tw_{\rep}(X, \D^+, g_0) \pmod{\Z}.
$$
\end{lemma}
\begin{proof}
We again consider $\xtilde$ as an end-periodic manifold with two ends, but now use a lift of the metric $g_0$ on $\xm$ and a lift of metric $g_1$ on $\xp$, with a metric $g$ on $W_0$ interpolating between the two. Using $f_- = -f$ and $f_+ = f$ for a choice of function $f$, we obtain
\begin{align*}
\ind_+ \D^+(\xtilde) &= \int_{W_0} \ii (\D^+(X,g))\;  \\
& - \int_Y\omega_1 + \int_X df \wedge\omega_1\; - \frac 1 2\,(h_1 + \eta(X, \D^+, f, g_1)) \\
& + \int_Y\omega_0 - \int_X df\wedge\omega_0\; - \frac 1 2\,(h_0 - \eta(X, \D^+, f, g_0))
\end{align*}
with a similar expression for $\ind_+ \D^+_{\rep}(\xtilde)$. Subtract $k$ times the first expression from the second and take into account the behavior of the local index forms and transgressed classes to conclude that
\[
\tw_{\rep}(X, \D^+, g_1) - \tw_{\rep}(X, \D^+, g_0)
\]
is  an integer. 
\end{proof}

\begin{remark}\label{R:invariance}
It is worth pointing out that the periodic $\tw$--invariant is also a diffeomorphism invariant, in the same sense as the classical $\tw$--invariant is; see Botvinnik--Gilkey \cite[page 516]{botvinnik-gilkey:bordism}. Let $H \subset \diff(X)$ denote the group of orientation-preserving diffeomorphisms $F: X \to X$ that preserve the class $\gamma \in H^1(X;\Z)$ and all spin structures.  Then for any $F \in H$,  any representation $\rep: \pi_1 (X) \to U(k)$, and any metric $g$ on $X$ for which $\tw$ is defined, we have 
\begin{equation}\label{E:invariance}
\tw_{\rep}(X, \D^+, g)\, =\, \tw_{F^*\rep}\, (X, \D^+, F^*g). 
\end{equation}
Note that the pull back of representations is only well defined on conjugacy classes (because of base point issues) but changing $\alpha$ within its conjugacy class does not affect $\tw_{\rep}$. 
\end{remark}

%%%%%%%%%%%%%%%%%%%%%%%%%%%%%%%%%%%%%%%%%%%%%%%

\subsection{Reduction to the classical $\tw$--invariant}
We saw in Section \ref{S:special} that the periodic $\tw$--invariant reduces to the classical $\tw$--invariant on product manifolds $X = S^1 \times Y$; see formula \eqref{E:rho-prod}. We will now show that a similar reduction holds for more general manifolds $X$.

Let $X$ have a metric $g$ which restricts to a product metric $g = d\theta^2 + g^Y$ on a product region $I \times Y \subset X$. Recall from Section \ref{S:product} that the operator $\D^+(X)$ has the form $d\theta\cdot(\p/\p \theta - \D(Y))$ on $I \times Y$. Similarly, for any representation $\alpha: \pi_1 (X) \to U(k)$, the operator $\D^+_{\rep}(X)$ has the form $d\theta\cdot(\p/\p \theta - \D_{\rep}(Y))$, where we use the same symbol $\rep$ to denote the restriction of $\alpha$ to $\pi_1 (Y)$.

\begin{proposition}\label{P:product-rho}
Let $X$ be a manifold as above, and suppose that the $L^2$--closures of the operators $\D^+(\tilde X)$ and $\D^+_{\rep}(\tilde X)$ are Fredholm.
\begin{enumerate}
\item If both operators $\D(Y)$ and $\D_{\rep}(Y)$ are invertible then 
\begin{equation}\label{E:rhoY}
\tw_{\rep}(X, \D^+, g)\;=\; \tw_{\rep}(Y, \D, g^Y) \pmod{\Z}.
\end{equation}
\item \label{l:psc}
 If both metrics $g$ and $g^Y$ have positive scalar curvature then~\eqref{E:rhoY} holds as an equality of real numbers.
\end{enumerate}
\end{proposition}

\begin{proof}
Consider an end-periodic manifold $\zp = ((-\infty,0] \times Y)\, \cup\, W_0 \, \cup\, \xp$. Because of the product region $I \times Y\subset X$, the metric $g$ on $X$ induces an obvious metric on $\zp$. Choose the function $f$ on $\xp$ so that $df$ has support in the product region. By Remark~\ref{R:product}, the terms involving the integrals of $\omega$ and $df\wedge \omega$ that appear in the index theorem all vanish. Using Theorem~\ref{T:aps} and the result of Section~\ref{S:product}, we obtain
$$
\ind \D^+(\zp) = \int_{W_0} \ii (\D^+(X,g)) -\frac12\left(\eta(X,\D^+)-\eta_{\D}(Y)\right),
$$
and a similar formula for $\ind \D_{\rep}^+(\zp)$. Formula \eqref{E:rhoY} follows by subtraction as in the proof of Lemma \ref{L:g-indep}. If $g$ has positive scalar curvature, the indices of $\D^+(\zp)$ and $\D^+_{\alpha} (\zp)$ vanish by Lemma~\ref{L:vanish}, and it follows that $\tw_{\rep}(X, \D^+, g) = \tw_{\rep}(Y, \D, g^Y)$.
\end{proof}

\begin{remark}
It is standard in the field~\cite{gilkey:sphere} to extend the definition \eqref{E:tw} to the ring $R_0 (\pi_1 (Y))$ of virtual unitary representations of $\pi_1 (Y)$ of virtual dimension zero. Given $\rep \in R_0 (\pi_1 (Y))$ of the form $\rep = \rep_1 - \rep_2$, where $\rep_1$ and $\rep_2: \pi_1 (Y) \to U(k)$ are representations, one defines
\[
\tw_{\rep} (Y,B)\; =\; \tw_{\rep_1} (Y,B)\, -\, \tw_{\rep_2} (Y,B).
\]
Our definition of the periodic $\tw$--invariant can similarly be extended to the ring $R_0 (\pi_1 (X))$ of virtual unitary representations of $\pi_1 (X)$ of virtual dimension zero. All properties of the $\tw$--invariant hold in this extended setting.
\end{remark}

%%%%%%%%%%%%%%%%%%%%%%%%%%%%%%%%%%%%%%%%%%%%%%%%%

\section{Metrics of positive scalar curvature}\label{S:psc}
One of the main applications of the $\tw$--invariant of the spin Dirac operator on odd-dimensional manifolds has been to the study of Riemannian metrics of positive scalar curvature (PSC for short). In this section, we will use our periodic $\tw$--invariant to study PSC metrics on certain even-dimensional manifolds. 

Denote the space of PSC metrics on a manifold $M$ by $\rrp(M)$. The quotient of $\rrp(M)$ by the group of self-diffeomorphisms of $M$ is called the moduli space of PSC metrics on $M$ and is denoted by $\mmp(M)$. These spaces may be empty: obstructions to the existence of PSC metrics on a manifold arise both from index theory~\cite{lichnerowicz:spinors,hitchin:spinors} and from minimal surface arguments~\cite{schoen-yau:psc}. However, if a manifold $M$ admits a PSC metric, refinements of these techniques show that the spaces $\rrp(M)$ and $\mmp(M)$ are often disconnected and may have non-trivial higher homotopy groups.  An introduction to the area may be found in~\cite{lawson-michelson}, and the papers~\cite{rosenberg-stolz:psc,rosenberg:psc-progress} survey some more recent results.  

For certain odd-dimensional spin manifolds $Y$, one can show that $\rrp(Y)$ and $\mmp(Y)$ have infinitely many components using the $\tw$--invariants associated to the spin Dirac operator on $Y$. The main result of this section is that our periodic $\tw$--invariants can be used to a similar effect for a class of $4n$--dimensional manifolds $X$: we will show that the spaces $\rrp(X)$ and $\mmp(X)$ may have infinitely many components when $n>1$, and arbitrarily many components when $n = 1$. The proofs will differ somewhat for $n > 1$ and $n = 1$ but both will rely on the result of the following section. 

%%%%%%%%%%%%%%%%%%%%%%%%%%%%%%%%%%%%%%%%%%%%%%%%%

\subsection{Homotopies through PSC metrics}
Throughout this section, $X$ will be a closed spin manifold of dimension $4n$ with a choice of a primitive class $\gamma \in H^1 (X;\Z)$, and $\D^+$ will denote a chiral spin Dirac operator on $X$.

\begin{theorem}\label{T:homotopy}
Suppose that $g_0$ and $g_1$ are metrics on $X$ with positive scalar curvature, so that the invariants $\tw_{\alpha} (X,\D^+,g_0)$ and $\tw_{\alpha} (X,\D^+,g_1)$ are defined for  any $\rep \in R_0(\pi_1 (X))$.  If $g_0$ and $g_1$ are homotopic through metrics of positive scalar curvature then
\begin{equation}\label{E:equal-rho}
\tw_{\rep}(X,\D^+,g_1)\, =\, \tw_{\rep}(X,\D^+,g_0).
\end{equation}
Moreover, let $F: X \to X$ be a diffeomorphism preserving $\gamma$ and spin structure, and such that $F^*\rep = \rep$ up to conjugation. If $g_0$ and $F^*g_1$ are homotopic through metrics of positive scalar curvature, then~\eqref{E:equal-rho} still holds.
\end{theorem}

\begin{proof}
We will treat the infinite cyclic cover $\xtilde \to X$ classified by $\gamma$ as an end-periodic manifold with two ends, as described in~\eqref{E:xtilde}. We claim that  for some $N > 0$, there exists a PSC metric $g$ on $\xtilde$ which is equal to $g_0$ on $\bigcup_{\,i \leq 0}\, W_i$ and to $g_1$ on $\bigcup_{\,i \geq N}\, W_i$.  Given this, the equality \eqref{E:equal-rho} of $\tw$--invariants will follow from Lemma~\ref{L:vanish} using a slight modification of the argument proving Lemma~\ref{L:g-indep}.

We now construct the metric $g$, starting with a family $g_t$, $0 \le t \le 1$, of PSC metrics on $X$ providing the homotopy between $g_0$ and $g_1$.  Fix an oriented, connected submanifold $Y \subset X$ whose homology class is dual to $\gamma$, and a lift of $Y$ to $\xtilde$ so that $Y$ is the right-hand boundary of $W_0$. Fix a smooth neighborhood $U = [-1,1] \times Y$ of $Y$ in $X$; the lift of $Y$ gives a lift of $U$ that we label $U_0$. The translates of $U_0$ by the covering translations $T^j$ will be labeled $U_j$; these overlap $W_j$ and $W_{j+1}$ as shown in Figure~\ref{F:xtilde}. Finally, fix a smooth function $\beta: U \to [0,1]$ such that $\beta = 0$ near $\{-1\} \times Y$ and  $\beta = 1$ near $\{1\} \times Y$; translation gives a similar function on each $U_j$.

\medskip

\begin{figure}[!ht]
  \begin{center}
    \labellist
    \normalsize\hair 0mm
    \pinlabel {{\footnotesize$g_{-1}$}} at 65 95
    \pinlabel {{\footnotesize$g_{0}$}} at 140 95
    \pinlabel {{\footnotesize$g_{1/N}$}} at 220 95
    \pinlabel {{\footnotesize$g_{2/N}\  \ldots\ldots$}} at 325 95
    \pinlabel {{\footnotesize$g_{1}$}} at 380 95
    \pinlabel {{\footnotesize$W_{-1}$}} at 65 -15
    \pinlabel {{\footnotesize$W_{0}$}} at 140 -15
    \pinlabel {{\footnotesize$W_1$}} at 220 -15
    \pinlabel {{\footnotesize$W_2\  \cdots\cdots$}} at 325 -15
    \pinlabel {{\footnotesize$W_N$}} at 380 -15
    \pinlabel {{\footnotesize$U_{0}$}} at 185 -10
    \pinlabel {{\footnotesize$U_1$}} at 265 -10
   \endlabellist
\includegraphics[scale=.65]{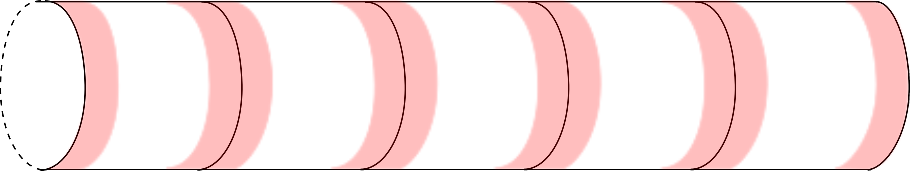}
  \end{center}
\caption{Decomposition of $\xtilde$}
\label{F:xtilde}
\end{figure}
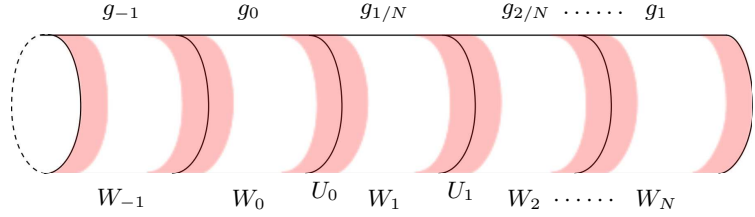

For an arbitrary $N >0$, consider the end-periodic metric on $\xtilde$ defined as follows. Start with an initial metric $\tilde{g}^{\,N}$ defined on individual $W_j$ by 
\begin{equation}
\tilde{g}^{\,N} (x) =
\begin{cases}
 \; g_0(x) &\mbox{if } x\in W_j,\ j<0\\ 
 \; g_{j/N}(x) &\mbox{if } x\in W_j, \; 0 \leq j \leq N\\
 \; g_1(x) &\mbox{if } x\in W_j,\ j>N.\\
\end{cases}
\end{equation}
These of course do not match up on the intersections $W_{j-1} \cap W_j$, so we modify them by defining the metric $g^N$ to be 
\begin{equation}\label{E:Uj}
(1-\beta(u)) \cdot g_{\,j/N} (u)\, +\, \beta(u) \cdot g_{\,(j+1)/N} (u)
\end{equation}
on $U_j$ for $j = 0,\ldots,N-1$. The metric $g^N$ is then well defined as an end-periodic metric on $\xtilde$.

On the complement of $U_j$, the metric $g^N$ equals $g_t$ for some $t \in [0,1]$ hence has positive scalar curvature by hypothesis. We will establish that, for $N$ sufficiently large, the metric defined by~\eqref{E:Uj} has positive scalar curvature throughout. This follows from two straightforward facts. First, the path $g_t$ (restricted to $U$) is uniformly continuous as a function from $[0,1]$ to the space $\met(U)$ of Riemannian metrics on $U$ with the $\cc^\infty$ topology. Therefore, if $N$ is sufficiently large, the metric defined by~\eqref{E:Uj} will be in any prescribed neighborhood of $g_{\,j/N}$ on $U$.  Second, the minimum (over $U$) of the scalar curvature is a continuous function of the metric. By compactness of $X$, the minimum of the scalar curvature of $g_t$ is bounded away from $0$, and this implies that for $N$ surficiently large, the scalar curvature of $g^N$ is positive everywhere on $\xtilde$.

The second statement of the theorem (concerning diffeomorphisms) follows from the first after taking Remark~\ref{R:invariance} into account. 
\end{proof}

%%%%%%%%%%%%%%%%%%%%%%%%%%%%%%%%%%%%%%%%%%%%%%%%%

\subsection{The $4n$--dimensional case with $n > 1$}
Our result will be an even-dimensional version of \cite[Theorem 0.3]{botvinnik-gilkey:bordism}, which holds for odd-dimensional closed spin manifolds $Y$ with a non-trivial finite fundamental group $G$.

\begin{theorem}\label{T:infinite}
Let $Y$ be a closed connected spin manifold of dimension $4n-1$ with $n> 1$ and with a non-trivial finite fundamental group $G$, and let $M$ be a closed spin manifold of dimension $4n$. If  both $Y$ and $M$ admit metrics of positive scalar curvature then $\pi_0\,(\mmp((S^1 \times Y)\conn M))$ is infinite.
\end{theorem} 

\begin{proof}
Start with the product metrics $d\theta^2 + g^Y_j$ on $S^1 \times Y$, where $g^Y_j$ is an infinite family of PSC metrics used in~\cite[Theorem 0.3]{botvinnik-gilkey:bordism} to prove the non-finiteness of $\pi_0(\mmp(Y))$ by showing that 
\[
\tw_{\rep}(Y,\D,g^Y_i)\; \neq\; \tw_{\rep}(Y,\D,g^Y_j)\quad \text{for}\quad i \neq j,
\]
for an explicitly constructed $\rep \in R_0(G)$. Note that the condition $r_m (G) > 0$ of that theorem is automatically satisfied because $m = 4n - 1 = 3\pmod 4$; see the remark following the statement of Theorem 0.1 in \cite{botvinnik-gilkey:bordism}.

Next, fix a PSC metric $g^M$ on $M$ and equip $X_M = (S^1 \times Y)\conn M$ with PSC metrics $g_j = (d\theta^2 + g^Y_j) \conn g^M$ constructed from PSC metrics on the two summands via connected sum using modifications supported in small neighborhoods of points in the summands. That this can be done follows from~\cite{gromov-lawson:psc,schoen-yau:psc}. Note that $Y$ continues to have a product neighborhood in $X_M$,
therefore, 
\[
\tw_{\rep}(X_M,\D^+,g_j)\; =\; \tw_{\rep}(Y,\D,g_j^Y)
\]

\smallskip\noindent
by part (2) of Proposition~\ref{P:product-rho}. Theorem~\ref{T:homotopy} then immediately implies that the metrics $g_j$ lie in different components of $\rrp(X_M)$, and indeed of $\rrp(X_M)/H$, where $H\subset \diff(X_M)$ is the group discussed in Remark~\ref{R:invariance}. Since $H$ is a subgroup of finite index, this implies as in Botvinnik--Gilkey~\cite{botvinnik-gilkey:bordism} that there are infinitely many path components in $\rrp(X_M)/\diff(X_M) = \mmp(X_M)$.
\end{proof}

\begin{remark}\label{R:break}
Note that the metrics $g^Y_j$ used in the above proof were constructed in \cite{botvinnik-gilkey:bordism} (based on \cite{miyazaki:psc,rosenberg:psc-novikov-II}) by pushing PSC metrics across a cobordism, which necessitates the hypothesis that $n>1$.
\end{remark}

\begin{remark}
Theorem~\ref{T:infinite} illustrates an important point about the periodic $\tw$--invariants.  By Proposition~\ref{P:product-rho}, when $Y \subset X$ has a metric product neighborhood, $\tw_{\rep}(X,\D^+,g)$ reduces to the classical $\tw$--invariant of the odd-dimensional manifold $Y$.  However, even if such a neighborhood is present for the PSC metrics $g_0$ and $g_1$, there is no reason it would be present for all PSC metrics in a homotopy $g_t$. Hence it is crucial for the proof of Theorem \ref{T:infinite} that we are able to define $\tw_{\rep}(X,\D^+,g)$ for arbitrary PSC metrics on $X$.
\end{remark}

%%%%%%%%%%%%%%%%%%%%%%%%%%%%%%%%%%%%%%%%%%%%%%%%%%

\subsection{The $4$--dimensional case}\label{S:4d}
The proof of Theorem \ref{T:infinite} for higher dimensional manifolds does not extend to manifolds of dimension 4 due to the break down in the topological arguments used to push PSC metrics across a cobordism; see Remark \ref{R:break}. These arguments, which go back to \cite{gajer:cobordism,gromov-lawson:psc,schoen-yau:psc}, do not work for surgeries along spheres of co-dimension one or two, and these are often unavoidable when dealing with manifolds of dimensions 4 and 5. Indeed, Seiberg--Witten theory shows~\cite{witten:monopole} that there are $4$-manifolds that do not carry PSC metrics, but are cobordant to manifolds that do carry such metrics.  One dimension down, Ricci flow techniques show~\cite{marques:psc} that $\pi_0(\mmp(Y))$ vanishes for any $3$-manifold $Y$ carrying a PSC metric.  

In this section, we will adapt the argument for Theorem~\ref{T:infinite} to produce a class of orientable $4$-manifolds $X$ for which $\pi_0(\mmp(X))$ may be arbitrarily large.  Non-orientable $4$-manifolds having this property were constructed in~\cite{ruberman:swpos}.

\begin{theorem}\label{T:4-dim}
Let $Y$ be a closed connected oriented $3$-manifold with non-trivial finite fundamental group, and let $M$ be a closed spin $4$-manifold which admits a metric of positive scalar curvature.  Then, for any positive integer $N$, there exists $m_N$ such that  
$
\pi_0 \left(\mmp((S^1 \times Y)\conn\, m_N\cdot (\sss) \conn M\right)
$
has at least $N$ elements.
\end{theorem} 

Before we go on to prove Theorem \ref{T:4-dim}, we need to review some basic handlebody theory~\cite{rourke-sanderson:book}.

\begin{lemma}\label{L:handles} 
Let $n \geq 4$, and suppose that $(W,X_0,X_1)$ is an $(n+1)$-dimensional cobordism with connected $X_0$ and $X_1$. Assume that the map $\pi_1 (X_0) \to \pi_1(W)$ is a surjection and the map $\pi_1 (X_1) \to \pi_1(W)$ is an isomorphism. Then $W$ has a handle decomposition relative to $X_0$ such that  
\begin{enumerate}
\item there are  no handles of index $0$ or $1$,\label{I:01}
\item there are no handles of index $n+1$ or $n$, and \label{I:n+1}
\item for each handle $h$ of index $n-1$, its belt sphere (or attaching circle for $h$ viewed as a  $2$-handle relative to $X_1$) is null-homotopic in $X_1$.\label{I:n-1}
\end{enumerate}
\end{lemma}

\begin{proof}
Item (1) is standard: $0$-handles are canceled by $1$-handles, and one can use the surjectivity of $\pi_1(X_0) \to \pi_1(W)$ to trade $1$-handles for $3$-handles as in~\cite[Lemma 6.15]{rourke-sanderson:book}. Turning over the handle structure yields item (2); note that one needs $n\geq 4$ to ensure that this step does not introduce any new $1$-handles. If item (3) failed to hold, the inclusion $\pi_1(X_1) \to \pi_1(W)$ would have non-trivial kernel.
\end{proof}

We also need a simple fact about the non-triviality of $\tw_\alpha(Y,\D,g^Y)$ in the $3$-dimensional case.  Although such calculations are well-known to experts in the field, we could not find the precise statement in the literature so we supply a quick proof here.

\begin{lemma}\label{L:space-form}
Let $Y$ be a $3$-dimensional spherical space form, with a spherical metric $g^Y$ (i.e. pushed down from $S^3$). Then for any spin structure on $Y$, there is a representation $\alpha: \pi_1(Y) \to U(k)$ such that $\tw_\alpha(Y,\D,g^Y) \neq 0$.
\end{lemma}

\begin{proof}
We will first prove this for lens spaces $L(p,q)$. Let $p > 1$ be an odd integer and view $L(p,q)$ as the quotient of the unit sphere $S^3 \subset \mathbb C^2$ by the cyclic group action
\[
t(z_1,z_2) = (e^{2 \pi i/p} z_1,e^{2 \pi i q/p} z_2).
\] 
Let $\alpha_1: \pi_1 (L(p,q)) \to U(1)$ be the representation sending $t$ to $e^{2 \pi i/p}$. Then, for the unique spin structure and the spherical metric $g$, we have
\[
\tw_{\alpha_1}(L(p,q),\D,g)\; =\; - (d/p) \cdot (p + 1)/2 \pmod \Z,
\]
where $d$ is a certain integer relatively prime to $48\,p$, see \cite[Theorem 2.5]{gilkey:eta-odd}. One can easily see that $\tw_{\alpha_1}(L(p,q),\D,g)$ is never zero modulo the integers, and hence it is not zero as a real number. The same theorem tells us that, for $L(2,1) = {\R}\rm{P}^3$, the invariants $\tw_{\alpha_1}(L(2,1),\D,g)$ are equal to $\pm 1/4 \pmod \Z$, depending on the $\spin$ structure.

Any other spherical space form $Y$ is finitely covered by a lens space $L(p,q)$ with a spin structure pulled back from $Y$. Let $\alpha$ be the representation of $\pi_1(Y)$ induced by the representation $\alpha_1$ on the finite index subgroup $\pi_1(L(p,q)) \subset \pi_1 (Y)$.  Then, according to~\cite[Lemma 2.5.6]{gilkey:sphere}, we have
$$
\tw_{\alpha}(Y,\D,g^Y)\; =\; \tw_{\alpha_1}(L(p,q),\D,g),
$$
which is not zero.
\end{proof}

%This line of argument gives considerably more, for instance that $\tw$ completely detects spin bordism over $B\pi_1(Y)$. 

\begin{proof}[Proof of Theorem~\ref{T:4-dim}]
Being a spherical space form, the manifold $Y$ admits a metric $g^Y$ of positive constant sectional curvature. Fix a spin structure on $Y$ and choose, via Lemma~\ref{L:space-form}, a representation $\alpha: \pi_1(Y) \to U(k)$ with $\tw_\alpha(Y,\D,g^Y) \neq 0$. Via the connected sum construction, the metric $g^Y$ gives rise to a PSC metric $g_0$ on the connected sum $m\cdot Y$ for any integer $m \ge 1$. Positive scalar curvature metrics proving the theorem will be constructed by pushing the product metric $d\theta^2 + g_0$ on $S^1 \times m\cdot Y$ across a carefully chosen cobordism, and using periodic $\tw$--invariants associated with $\alpha$ to distinguish their moduli.

We begin our construction of the cobordism with the observation that, due to the finiteness of $\pi_1 (Y)$, the spin cobordism group $\Omega_3^{\spin}(B\pi_1(Y))$ is finite, hence there is a positive integer $d$ annihilating its every element. It then follows that, for any $n$, there is a spin cobordism $V_n$ and a representation $\tilde\alpha: \pi_1(V_n) \to U(k)$ such that
\[
\p\, (V_n,\tilde\alpha)\; =\; (Y,\alpha)\; -\; (nd +1)\cdot (Y,\alpha),
\]
where $r\cdot (Y,\alpha)$ stands for the connected sum of $r$ copies of $Y$ with representation $\alpha$ on each summand. One may further assume, after killing the kernel of the map $\pi_1 (V_n) \to \pi_1 (Y)$ by surgery on some circles in $V_n$, that the inclusion of $Y$ into $V_n$ induces an isomorphism on the fundamental groups, and that the same is true for each summand of $(nd + 1)\cdot Y$.

Next, endow $S^1$ with a non-bounding spin structure and consider the spin cobordism $W_n = S^1 \times V_n$ 
with boundary
\[
\p\,W_n\; = \; S^1 \times Y\; - \; S^1 \times (nd + 1)\cdot Y.
\]
Since this cobordism satisfies the hypotheses of Lemma~\ref{L:handles}, we will assume that $W_n$ has a handle decomposition with only $2$- and $3$-handles. Let $k_n$ be the number of the $3$-handles, and view them as $2$-handles relative to $S^1 \times Y$.  By item \eqref{I:n-1} of Lemma \ref{L:handles}, the attaching maps of these $2$-handles are null-homotopic in $S^1 \times Y$, so the result of adding them is a spin cobordism from $(S^1 \times Y) \conn\, k_n\cdot (\sss)$ to $S^1\times Y$. Note that we are taking the connected sum with $\sss$ rather than the non-trivial $S^2$-bundle over $S^2$, which is not a spin manifold.
  
Attaching the original $2$-handles of $W_n$ to $S^1 \times (nd + 1)\cdot Y$ then results in a spin cobordism $U_n$ with boundary 
\[
\p\, U_n\; = \; (S^1 \times Y) \conn\, k_n\cdot (\sss)\; - \; S^1 \times (nd + 1)\cdot Y.
\]
One can easily check that the inclusion of $(S^1 \times Y) \conn\, k_n\cdot (\sss)$ into $U_n$ induces an isomorphism on the fundamental groups, and so there is a surjection $\phi: \pi_1(U_n) \to \Z$ whose restriction to the other boundary component of $U_n$ is the obvious projection $\pi_1(S^1 \times (nd + 1)\cdot Y) \to \pi_1(S^1) =  \Z$. Under this projection, the attaching circle of any $2$-handle must map to zero because $H_1(U_n;\Q) = \Q$. It then follows that this attaching circle is homotopic to a curve in $(nd + 1)\cdot Y$. Since homotopy implies isotopy in dimension four, we may assume that all of the attaching circles of the $2$-handles of $U_n$ (relative to $S^1 \times (nd + 1)\cdot Y$) live outside of some product region $I \times (nd + 1)\cdot Y$. Since $U_n$ has only $2$-handles, the construction of Gajer~\cite{gajer:cobordism} gives it a PSC metric, which extends the product metric on $S^1 \times (nd + 1)\cdot Y$ and which is a product metric near both boundary components. The construction involves modifying the metric on $S^1 \times (nd + 1)\cdot Y$ in some neighborhood of the attaching maps of the $2$-handles, and hence does not affect the metric in the aforementioned product region. 

Let $g$ be the metric obtained by restricting this metric to $(S^1 \times Y) \conn k_n\cdot (\sss)$. It has positive scalar curvature and also contains a product neighborhood of $(nd + 1)\cdot Y$. Using part \eqref{l:psc} of Proposition~\ref{P:product-rho} and formula \eqref{E:rho-prod}, we calculate
\[
\tw_\alpha((S^1 \times Y) \conn k_n\cdot (\sss),\D^+,g)\; =\; \tw_\alpha\,((nd + 1)\cdot Y,\D,g_0).
\]
To calculate the latter $\tw$--invariant, consider a standard spin cobordism between the disjoint union of $nd + 1$ copies of $Y$ and the connected sum $(nd + 1)\cdot Y$ obtained by adding $1$-handles to the disjoint union. According to Gajer~\cite{gajer:cobordism}, this cobordism has a metric of positive scalar curvature which restricts to a product metric near its boundary. The representation $\rep$ extends over this cobordism, hence it follows from~\cite{aps:II} that
\[
\tw_{\rep} ((nd + 1)\cdot Y,\D,g)\; =\; (nd + 1)\cdot \tw_{\rep} (Y,\D,g^Y).
\]
Since $\sss$ has a metric of positive scalar curvature, the connected sum construction gives us a natural PSC metric on $(S^1 \times Y) \conn k\cdot(\sss)$ with
\[
\tw_\alpha((S^1 \times Y) \conn k \cdot (\sss),\D^+,g)\; = \; (nd + 1)\cdot \tw_{\rep}\,(Y,\D,g^Y)
\]
 for any $k \geq k_n$. These $\tw$--invariants are distinct for different $n$ because $\tw_{\rep}\,(Y,\D,g^Y)\neq 0$.
  
To complete the proof of the theorem, let $N$ be a positive integer, and let $m_N$ be the maximum of  $\{k_n\, |\, n=1,\ldots,N\}$. Then by adding more copies of $\sss$ if necessary to the manifolds $(S^1 \times Y) \conn k_n\cdot (\sss)$, we obtain metrics on $(S^1 \times Y) \conn m_N\cdot (\sss)$ with $\tw$--invariants 
\[
(nd + 1)\cdot \tw_\alpha(Y,\D,g^Y),\quad n=1,\ldots,N.
\]
 As in the proof of Theorem~\ref{T:infinite}, these metrics are distinct up to homotopy, even after connected sum with an arbitrary spin manifold $M$ carrying a PSC metric. 
\end{proof}

%%%%%%%%%%%%%%%%%%%%%%%%%%%%%%%%%%%%%%%%%%%%%%%%%

\section{Heat kernel estimates}\label{S:estimates}
Let $M$ be a Riemannian manifold of dimension $n$, and $\Sm$ a Dirac bundle with associated Dirac operator $\D = \D(M)$.  In this section, we derive estimates on the smoothing kernel of the operator $\exp(-t\D^2)$, as well as estimates on its derivatives.  These were needed in the proof of Theorem~\ref{T:aps} but were postponed for the sake of the exposition.  In the first two subsections, we make only the assumption (cf. Roe~\cite[Section 2]{roe1}) that the pair $(M;\Sm)$ has bounded geometry.  This means that the injectivity radius of $M$ is bounded from below, and that the norm of its curvature tensor (and its covariant derivatives) is bounded from above. Similarly, the curvature of the Clifford connection (along with its covariant derivatives) on $\Sm$ has bounded norm. Of course, end-periodic manifolds and end-periodic Dirac operators satisfy these conditions. The last three subsections will be specific to operators on end-periodic manifolds. We will work most of the time with full Dirac operators; their chiral counterparts can be treated in similar fashion.

Heat kernel estimates similar to ours but in the context of relative index theory can be found in~\cite{bunke:relative}; see also~\cite{donnelly:spectrum} and~\cite{cheng-li-yau:kernel} for the case of scalar Laplacian. We have chosen to give a full treatment here because we need stronger results regarding the behavior of heat kernels at
$t \to \infty$ than the relative index theory provides; in addition, there are essential differences with~\cite{bunke:relative} in how we arrive at our estimates, including our use of gradient estimates and of Taubes' trick~\cite{taubes:spectral} for converting estimates for the scalar heat kernel to such estimates for more general operators.

%%%%%%%%%%%%%%%%%%%%%%%%%%%%%%%%%%%%%%%%%%%%%%%%

\subsection{Smoothing kernels}\label{S:smooth}
The paper of Roe~\cite{roe1} explains the basic analytical properties of Dirac operators $\D = \D(M)$ that hold whenever $(M;\Sm)$ has bounded geometry. Most important for us is the construction of the smoothing kernel for operators of the form $h(\D)$ for $h$ a rapidly decaying function.   We briefly summarize some properties we will use, referring to~\cite{roe1} for more details.  

Let $h: \R \to \R$ be a continuous function with the property that for each integer $k \ge 0$ there exists a constant $C_k$ such that $|h(s)| \le C_k\,(1 + |s|)^{-k}$. Then the operator $h(\D)$ defined by the spectral theorem can be represented by its smoothing kernel $K(x,y)$ as in \eqref{E:h(D)}. With respect to topologies described on pages 93--94 of \cite{roe1}, the map that associates its smoothing kernel to such an operator is continuous; see Proposition 2.9 of \cite{roe1}. In particular, the operators $\D^m\,\exp(-t \D^2)$ are represented by such smoothing kernels for all $t > 0$ and $m \ge 0$. 

%%%%%%%%%%%%%%%%%%%%%%%%%%%%%%%%%%%%%%%%%%%%%%%%

\subsection{Estimates for the kernel of $\exp(-t \D^2)$}
In this subsection, we establish estimates, valid whenever $(M;\Sm)$ has bounded geometry, for the smoothing kernel of the operator $\exp(-t \D^2)$. This kernel will be denoted $K = K(t;x,y)$ and viewed as a smooth section of the bundle $\Hom (\pi_R^* S,\pi_L^* S)$ over $(0,\infty) \times M \times M$. As a function of $t$ and $x$ with $y$ fixed, it solves the initial value problem
\begin{equation}\label{E:one}
\left(\ddt + \D^2\right) \,K = 0,\quad \lim_{t\to 0}\; K\;=\;
\delta_y\cdot\I,
\end{equation}
where $\mathbb I$ is the identity automorphism. 

We begin with short-term Gaussian estimates on $K$ and its derivatives.

\begin{proposition}
Let $(M,\Sm)$ have bounded geometry and let $K(t;x,y)$ be the smoothing kernel 
of the operator $\exp (-t\D^2)$ on $M$. Then for any $T > 0$, there is a 
positive constant $C$ such that
\begin{equation}\label{E:short}
\left|\frac{\p^i}{\p t^i}\,\nabla_x^j\,\nabla_y^k\; K(t;x,y)\right|\; \le\; 
C\,t^{-n/2 - i - |j| - |k|}\,e^{-d^2(x,y)/4t},
\end{equation}
for all $t \in (0,T\,]$. Here, $j$ and $k$ are multi-indices, and the constant 
$C$ only depends on $T$.
\end{proposition}

\begin{proof}
The proof is essentially the same as that for scalar heat kernels; see for 
instance Donnelly \cite[Section 3]{don1} and Donnelly \cite[Section 4]{don2}.
\end{proof}

Note that, for any $m \ge 0$, the smoothing kernels of $\D^m \exp(-t \D^2)$ are of the form $\D^m K(t;x,y)$, where $\D$ is the Dirac operator acting on the $x$--variable. In particular, we see that estimates similar to \eqref{E:short} hold as well for the smoothing kernels of the operators $\D^m \exp(-t \D^2)$ and their chiral versions.

We turn next to Gaussian estimates on $|K(t;x,y)|$, valid for all $t > 0$, using the well known results of Li and Yau~\cite[Corollary 3.1]{li-yau} on the scalar heat kernel and Taubes' trick~\cite[Proposition 2.1]{taubes:spectral}. More precisely, we will prove the following.

\begin{proposition}\label{P:longterm}
Let $(M;\Sm)$ have bounded geometry and let $K(t;x,y)$ be the smoothing 
kernel of the operator $\exp (-t\D^2)$ on $M$. There exist positive constants 
$\alpha$, $\gamma$, and $C$ such that 
\begin{equation}\label{E:long}
|K(t;x,y)|\;\le\;C\,e^{\alpha t}\,t^{-n/2} e^{-\gamma d^2(x,y)/t}\quad
\text{for all\; $t > 0$}.
\end{equation}
\end{proposition}

The rest of this subsection will be dedicated to the proof of this proposition. We begin with the generalized Bochner formula~\cite[Chapter II,\S 8]{lawson-michelson},
\[
\D^2 = \nabla^* \nabla + \rr,
\]
where $\rr$ is defined in terms of the curvature of the connection on $\Sm$. In the special case when $M$ is a spin manifold and $\D$ the spin Dirac operator, $\rr$ is just $1/4$ times the scalar curvature of $M$. Plugging this into \eqref{E:one}, 
we obtain 
\[
\frac{\p K}{\p t} + \nabla^* \nabla (K) + \rr \cdot K = 0
\]
and 
\begin{equation}\label{E:two}
\left(\frac{\p K}{\p t}\,, K\right) + (\nabla^* \nabla (K),K) + 
(\rr\cdot K,K) = 0.
\end{equation}

\smallskip\noindent
Here, the parentheses stand for the fiberwise inner product on the vector bundle $\Hom (\pi_R^* S,\pi_L^* S)$. Let $\Delta = d^* d =  - *d*d$ be the scalar Laplace operator.

\begin{lemma}\label{L:one}
For any section $s$ of a Euclidean bundle with the compatible connection $\nabla$ one has 
\[
\Delta (|s|^2)\, =\, 2 (\nabla^* \nabla s,s) - 2 |\nabla s\,|^2.
\]
\end{lemma}

\noindent
Using Lemma \ref{L:one} with $s = K$, the formula \eqref{E:two} can easily be converted into 
\[
\frac 1 2\cdot\frac {\p\,|K|^2}{\p t} + \frac 1 2\, \Delta (|K|^2) + 
|\nabla (K)|^2 + (\rr\cdot K,K)  = 0,
\]
which after another application of Lemma \ref{L:one} with $s = |K|$ becomes
\begin{equation}\label{E:three}
|K|\cdot \frac{\p\,|K|}{\p t} + |K|\cdot\Delta(|K|) - 
|\,d|K|\,|^2 + |\nabla K|^2 +  (\rr\cdot K,K)= 0.
\end{equation}

\begin{lemma}\label{L:two}
For any section $s$ of a Euclidean vector bundle with the compatible connection $\nabla$ one has $|\,d|s|\,| \le |\nabla s|$.
\end{lemma}

\begin{proof}
This is known as Kato's inequality; see for instance formula (3.3) in Taubes 
\cite{taubes:stable}.
\end{proof}

\begin{proof}[Proof of Proposition~\ref{P:longterm}] Together, Lemma \ref{L:two} and formula \eqref{E:three} yield the differential inequality
\[
\frac{\p\,|K|}{\p t} + \Delta(|K|)\; \le\;  \|\rr\|\;|K|.
\]
Since $(M;\Sm)$ has bounded geometry, the curvature operator $\rr$ is bounded, and hence there is a constant $\alpha \ge 0$ such that
\[
\frac{\p\,|K|}{\p t} + \Delta(|K|)\; \le\; \alpha\,|K|.
\]

For any fixed $y \in M$, let us consider the function $h(t,x) = e^{-\alpha t}
\,|K|$. A straightforward calculation shows that $h$ satisfies the 
differential inequality
\[
\frac{\p h}{\p t} + \Delta (h)\; \le\; 0
\]
with the initial condition 
\[
\lim_{t \to 0}\; h\, =\, \lim_{t \to 0}\; |K|\, =\, |\delta_y\cdot \mathbb I| = 
\kappa\cdot \delta_y
\]
for some positive constant $\kappa$. Let $H(t;x,y)$ be the scalar heat kernel, that is, the smoothing kernel of the operator $\exp(-t \Delta)$. As a function of $t$ and $x$ with $y$ fixed, $H$ solves the initial value problem 
\[
\frac{\p H}{\p t} + \Delta (H) = 0,\quad \lim_{t\to 0}\; H\;=\;\delta_y.
\]
Then the difference $k = h - \kappa\cdot H$, as a function of $t$ and $x$, 
solves the initial value problem
\[
\frac{\p k}{\p t} + \Delta (k) \le 0,\quad \lim_{t\to 0}\; k\;=\;0.
\]

\medskip
The maximum principle can be applied to $k$ even though the manifold $M$ is 
not compact because estimates \eqref{E:short} and similar estimates for 
$H(t;x,y)$ ensure that, for any fixed $t$ and $y$, the function $k$ 
approaches zero when $x$ runs off to infinity. The maximum principle
implies that $k (t,x) \le 0$ for all $t > 0$ and $x \in M$, which of course 
translates into the inequality
\[
|K(t;x,y)|\;\le\;\kappa e^{\alpha t}\cdot H(t;x,y).
\]
Now the Gaussian estimates on the scalar heat kernel $H(t;x,y)$ found in 
\cite[Corollary 3.1]{li-yau} complete the proof. 
\end{proof}

%%%%%%%%%%%%%%%%%%%%%%%%%%%%%%%%%%%%%%%%%%%%%%%%%

\subsection{Long-term derivative estimates}
The results of this subsection are specific to the periodic manifold 
$\tilde X$ and do not necessarily extend to general manifolds of bounded 
geometry.  

\begin{proposition}
Let $\tilde K(t;x,y)$ be the smoothing kernel of the operator $\exp(-t \D^2)$
on $\tilde X$. There exist positive constants $\alpha$, $\gamma$, and $C$ 
such that 
\begin{equation}\label{E:grad}
|\nabla \tilde K(t;x,y)|\;\le\;C\,e^{\alpha t}\,t^{-n/2-1} e^{-\gamma d^2(x,y)/t}
\quad\text{for all\; $t > 0$}.
\end{equation}
\end{proposition}

\begin{proof} 
Differentiate equation \eqref{E:one} with respect to $t$ to conclude that 
$\tilde K' = \p \tilde K/\p t$ solves the equation
\[
\left(\frac {\p}{\p t} + \D^2\right) \tilde K' = 0
\]
with the initial condition 
\[
\lim_{t\to 0}\, \tilde K'\; = \; - \lim_{t \to 0}\, \D^2\,\tilde K\; = \; 
- \D^2\,(\delta_y\cdot\I)\; = \; - (\Delta\,\delta_y)\cdot \I.
\]
Similarly, the time derivative $\tilde H'$ of the scalar heat kernel 
$\tilde H$ on $\tilde X$ solves the initial value problem
\[
\frac {\p \tilde H'}{\p t} + \Delta (\tilde H') = 0,\quad \lim_{t \to 0}\, 
\tilde H'\; =\; - \Delta\,\delta_y.
\]
The argument of the previous section can now be applied to the time 
derivatives of $\tilde K$ and $\tilde H$ to deduce that 
\[
|\tilde K'(t;x,y)|\;\le\;\kappa e^{\alpha t}\,\tilde H'(t;x,y)
\]
for some positive constants $\kappa$ and $\alpha$. The Gaussian estimates on the time derivatives of $\tilde H$; see \cite[Theorem 3]{davies} and also \cite{grigoryan} ensure that there exist positive constants $\gamma$ and $C$ such that  
\[
|\tilde K' (t;x,y)|\;\le\;C\,e^{\alpha t}\,t^{-n/2 -1}e^{-\gamma d^2(x,y)/t}
\quad\text{for all\; $t > 0$}.
\]
Using this estimate, one can argue as in \cite[Lemma 2.3]{coulhon-duong} 
that there are positive constants $\beta$ and $C$ such that
\[
\int_{\tilde X} e^{\,\beta\,d^2(x,y)/t}|\nabla \tilde K(t;x,y)|^2\,dx\;\le\;
C\,e^{2\alpha t}\,t^{-n/2 - 1}.
\]
With this weighted $L^2$--estimate in place, one can  follow the argument 
of \cite{dungey} to derive the pointwise estimates \eqref{E:grad}. The 
caveat is that both \cite{coulhon-duong} and \cite{dungey} deal with 
scalar heat kernels but the aforementioned arguments go through with 
little change to cover the case of $\tilde K(t;x,y)$.
\end{proof}

%%%%%%%%%%%%%%%%%%%%%%%%%%%%%%%%%%%%%%%%%%%%%%%%

\subsection{On-diagonal estimates}
Let $\zp = Z\,\cup\,\txp$ be a manifold with periodic end, where $\txp = W_0\,\cup\,W_1\,\cup\ldots$. Let $K(t;x,y)$ and $\tilde K(t;x,y)$ be the smoothing kernels of the operators $\exp(-t \D^2)$ on, respectively, $\zp$ and $\tilde X$. This subsection is devoted to the proof of the following result. 

\begin{proposition}\label{P:diff}
There are positive constants $\alpha$, $\gamma$ and $C$ such that, for 
all $t > 0$ and all $x \in W_k$ with $k \ge 1$, one has
\[
|K(t;x,x) - \tilde K(t;x,x)|\;\le\; C\,e^{\,\alpha t}\,e^{-\gamma\,d^2(x,W_0)/t}.
\]
\end{proposition}

\begin{proof} The proof will rely on the construction of the heat kernel on $\zp$ via the Duhamel principle, cf. \cite[Section 22C]{BW} in the product end case. We will use the intersection $(Z\,\cup\,W_0)\,\cap\,\txp = W_0$ as the gluing region for patching the heat kernel on $DZ = Z\,\cup\, W_0\,\cup\, (-W_0)\,\cup\, (-Z)$ with that on $\tilde X$. 

Let $h: \tilde X \to \R$ be a smooth function such that $h(x+1) = h(x) + 1$,
$h(W_0) \subset [0,1]$ and $h$ equals zero on $\p_- W_0 = -Y$ and one on 
$\p_+ W_0 = Y$. The restriction of $h$ to $\txp$ is nowhere negative; 
we extend it to a smooth function on $\zp = Z\,\cup\,\txp$ (called again 
$h$) so that it is negative on the interior of $Z$. For any real numbers 
$a < b$, let $\rho_{a,b}$ be an increasing smooth function of real variable 
$u$ such that
\[
\rho_{a,b}\, (u) = 
\begin{cases} 
\;0, & \quad\text{for $u \le a$}, \\
\;1, & \quad\text{for $u \ge b$}.
\end{cases}
\]
Define smooth cut-off functions $\phi_1, \phi_2, \psi_1$, $\psi_2: \zp \to [0,1]$ by defining them first on $W_0$ by the formulas 
\begin{alignat*}{2}
&\phi_1 = 1 - \rho_{5/7,6/7} \circ h, &\qquad &\phi_2 = \rho_{1/7,2/7} \circ h, \\
&\psi_1 = 1 - \rho_{3/7,4/7} \circ h, &\qquad &\psi_2 = 1 - \psi_1,
\end{alignat*}
and then extending to the entire $\zp$ by 0 or 1 in an obvious way; see the schematic picture below. The functions $\psi_1$ and $\psi_2$ form a partition of unity subordinate to the open covering $\zp = \{\,h(x) < 5/7\,\}\,\cup\,\{\,h(x) > 2/7\,\}$. In addition, $\phi_j = 1$ on $\supp \psi_j$ and the distance between $h(\supp \nabla\phi_j)$ and $h(\supp \psi_j)$ is no less than $1/7$ for both $j = 1$ and $j = 2$.

\bigskip

\begin{figure}[!ht]
\centering
\psfrag{V}{$V$}
\psfrag{V1}{$V_1$}
\psfrag{V2}{$V_2$}
\psfrag{U}{$W_0$}
\psfrag{U1}{}
\psfrag{U2}{}
\psfrag{psi1}{$\psi_1$}
\psfrag{psi2}{$\psi_2$}
\psfrag{phi2}{$\phi_2$}
\psfrag{phi1}{$\phi_1$}
\psfrag{Z}{$Z$}
\psfrag{Y}{}
\psfrag{W0}{}
\psfrag{W1}{$W_1$}
\includegraphics[scale=1.0]{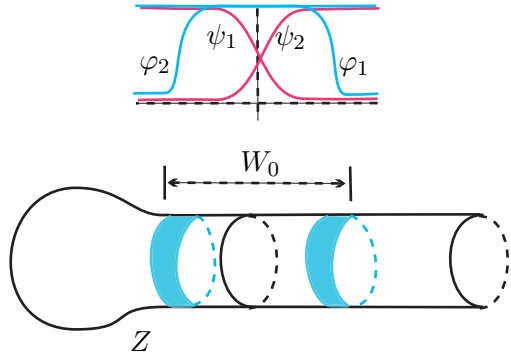}
\caption{Bump functions}
\label{F:bump}
\end{figure}

Let $K_1 (t;x,y)$ and $K_2 (t;x,y) = \tilde K(t;x,y)$ be smoothing 
kernels of the operators $\exp(-t \D^2)$ on $DZ$ and $\tilde X$, 
respectively. Define an approximate smoothing kernel $K^a (t;x,y)$ of 
$\exp(-t \D^2)$ on $\zp$ by the formula
\begin{equation}\label{E:approx}
K^a (t;x,y) = \sum_{j = 1}^2\;\phi_j (x)\,K_j(t;x,y)\,\psi_j (y).
\end{equation}
Note that $K^a (t;x,y) = K_1 (t;x,y)$ when $(x,y) \in Z\times Z$ and $K^a (t;x,y) = K_2 (t;x,y)$ when $(x,y) \in (\txp - W_0)\times (\txp - W_0)$ and also that $K^a (t;x,y) = 0$ when $(x,y) \in Z \times (\txp - W_0)$ or $(x,y) \in (\txp - W_0)\times Z$. In addition, $\lim K^a (t;x,y) = \delta_y (x)\cdot \mathbb I$ when $t \to 0$ for all $(x,y)$. This means that $K^a$ solves the initial value problem \eqref{E:one} for most $(x,y) \in \zp \times \zp$. To be precise, let us consider the error term induced by the approximate smoothing kernel $K^a (t;x,y)$ on $\zp$,
\begin{equation}\label{E:error}
- E(t;x,y) = \left(\ddt + \D^2\right) K^a (t;x,y), 
\end{equation}
where the operator $\D$ acts on the $x$-variable for any fixed $t$ and $y$. 

\begin{lemma}
The error term $E(t;x,y)$ vanishes unless $h(x) \in (1/7, 6/7)$ and the distance between $h(x)$ and $h(y)$ is greater than $1/7$. There are positive constants $\alpha$, $\gamma$ and $C$ such that, for all $t > 0$ and all $x, y \in \zp$, the following estimate holds
\begin{equation}\label{E:err}
|E(t;x,y)|\; \le \; C\,e^{\,\alpha t}\,t^{-n/2-1} e^{-\gamma\,d^2(x,y)/t}.
\end{equation}
\end{lemma}

\begin{proof}
Apply the formula of \cite[Lemma 7.13]{Roe} to the spinor $\phi_j\,K_j$ to obtain 
\[
\D^2 (\phi_j\,K_j) = (\Delta\,\phi_j) K_j - 2 \nabla_{\nabla \phi_j} K_j +
\phi_j\,\D^2 K_j.
\]

\smallskip\noindent
Since both $K_j$ satisfy \eqref{E:one} we calculate 
\smallskip
\[
- E(t;x,y) = \sum_j\; (\Delta \phi_j(x) K_j(t;x,y) \psi_j(y) - 
2\,\nabla_{\nabla \phi_j(x)} K_j(t;x,y) \psi_j(y)).
\]
The claim now follows by applying \eqref{E:long} to $K_1$ and $K_2$ and \eqref{E:grad} to $\nabla K_2$, and using the standard estimates on $\nabla K_1$ on the closed manifold $DZ$.
\end{proof}

Denote by $\K(t) = \exp(-t \D^2)$ and $\K^a(t)$ the operators on $\zp$ with smoothing kernels $K(t;x,y)$ and $K^a (t;x,y)$, respectively. Because of the initial conditions $\K(t) \to \I$ and $\K^a (t) \to \I$ as $t \to 0$, we can write at the operator level
\begin{multline}\notag
\K(t) - \K^a(t) 
=\int_0^t \frac d{ds}\,\left(\K(s)\cdot\K^a(t-s)\right)\,ds \\
=\int_0^t \frac d{ds}\,\K(s)\cdot \K^a(t-s)\,ds + 
\int_0^t \K(s)\cdot\frac d{ds}\, \K^a(t-s)\,ds.
\end{multline}
Since
\[
\frac d{ds}\, \K(s) = - \D^2\, \K(s) = - \K(s)\, \D^2,
\] 

\medskip\noindent
the above can be written as
\begin{multline}\label{E:diff0}
\K(t) - \K^a(t) 
= \int_0^t \K(s)\cdot \left(-\D^2 + \frac d{ds}\right) \K^a(t-s)\,ds = \\
\int_0^t \K(s)\cdot \left(-\D^2 - \frac d{d(t-s)}\right) \K^a(t-s)\,ds
= \int_0^t \K(s)\cdot \E (t-s)\,ds,
\end{multline}

\medskip\noindent
where $\E(t)$ is the operator with the smoothing kernel $E(t;x,y)$. At the
level of smoothing kernels, formula \eqref{E:diff0} implies that
\[
K(t;x,x) - K^a(t;x,x) = \int_0^t \int_{\zp} K(s;x,z)\,E(t-s;z,x)\,dz\,ds.
\]
The $z$--integration in this formula extends only to $\supp_z E(t-s;z,x)
\subset \{\,z \in W_0\;|\;1/7 \le h(z) \le 6/7\,\}$. In particular, there 
exists $\ep > 0$ such that $\supp_z E(t-s;z,x) \subset N$, where $N = 
\{\,z \in W_0\;|\;d\,(z,\p W_0) \ge \ep\,\}$, and 
\begin{equation}\label{E:diff}
K(t;x,x) - K^a(t;x,x) = \int_0^t \int_N K(s;x,z)\,E(t-s;z,x)\,dz\,ds.
\end{equation}
Let $z \in N$ and restrict ourselves to $x \in W_k$ with $k \ge 1$. Then $\ep^2 + d^2(x,W_0)\;\le 
\;d^2(x,z)$ and we have the estimates
\[
|K(s;x,z)|\;\le\; C\,e^{\,\alpha s} s^{-n/2} e^{-\gamma\,\ep^2/s}\, 
e^{-\gamma\,d^2 (x,W_0)/s}\;\le\;C_1\,e^{\,\alpha s}\,e^{-\gamma\,d^2 (x,W_0)/s}
\]
for all $s > 0$ (see \eqref{E:long}) and 
\begin{multline}\notag
|E(t-s;z,x)|\;\le\;C\,e^{\,\alpha (t-s)} (t-s)^{-n/2-1} e^{-\gamma\,\ep^2/(t-s)} 
e^{-\gamma\,d^2 (x,W_0)/(t-s)}\\
\le\;C_2\,e^{\,\alpha (t-s)} e^{-\gamma\,d^2 (x,W_0)/(t-s)}
\end{multline}
for all $s \in (0,t)$ (see \eqref{E:err}). Combining these estimates with \eqref{E:diff}, we obtain
\begin{multline}\notag
|K(t;x,x) - K^a(t;x,x)| \\ \le\;
C_3\,e^{\,\alpha t}\,\int_0^t \int_N e^{-\gamma\,d^2 (x,W_0)(1/s + 1/(t-s))}\,dz\,ds 
\;\le\;C_4\,e^{\,\alpha t}\,e^{-\gamma\,d^2 (x,W_0)/t},
\end{multline}
where we used the inequality $1/t \le 1/s + 1/(t-s)$, which obviously holds 
for all $s \in (0,t)$. This completes the proof of Proposition \ref{P:diff}. 
\end{proof}

\begin{corollary}\label{C:diff}
For any given $T > 0$, there are positive constants $\gamma$ and $C$ such that, for all $t \in (0,T]$ and all $x \in W_k$ with $k \ge 1$, one has
\[
|K(t;x,x) - \tilde K(t;x,x)|\;\le\; C\,e^{-\gamma\,d^2(x,W_0)/t}.
\]
\end{corollary}

The following result is a slight generalization of the above corollary; it 
provides a good estimate on $|K(t;x,y) - \tilde K(t;x,y)|$ when $(x,y)$ 
is not necessarily on the diagonal but sufficiently close to it.

\begin{proposition}\label{P:diff2}
For any given $T > 0$, there are positive constants $\gamma$ and $C$ such that, for all $t \in (0,T]$ and all $x \in W_k$ and $y \in W_{\ell}$ with $k, \ell \ge 1$, one has 
\[
|K(t;x,y) - \tilde K(t;x,y)|\;\le\; C\,e^{-\gamma\,d^2/t},
\]
where $d^2$ is the minimum of $d^2(x,W_0)$ and $d^2(y,W_0)$.
\end{proposition}

\begin{proof}
We will essentially follow the proof of Proposition \ref{P:diff} with the 
factor $e^{\,\alpha t}$ replaced by a constant on the bounded time interval. 
For any $x$ and $y$ as in the statement of the proposition, formula 
\eqref{E:diff0} implies that 
\begin{equation}\label{E:diff3}
K(t;x,y) - K^a(t;x,y) = \int_0^t \int_N K(s;x,z)\,E(t-s;z,y)\,dz\,ds,
\end{equation}
compare with \eqref{E:diff}. We can estimate 
\[
|K(s;x,z)|\;\le\; C_1 e^{-\gamma\,d^2 (x,W_0)/s},\quad
|E(t-s;z,y)|\;\le\; C_2 e^{-\gamma\,d^2 (y,W_0)/(t-s)}
\]
and use the obvious inequality
\[
-d^2(x,W_0)/s - d^2(y,W_0)/(t-s) \le -d^2\,(1/s + 1/(t-s)) \le - d^2/t
\]
to arrive at the desired estimate.
\end{proof}

\begin{remark}\label{R:diff}
For any integer $m \ge 0$ the statements of Corollary \ref{C:diff} and 
Proposition \ref{P:diff2} also hold if $K(t;x,y)$ and $\tilde K(t;x,y)$ 
are the smoothing kernels of the operators $\D^m\exp(-t \D^2)$ on 
respectively $\zp$ and $\tilde X$. The above proofs work with little 
change once we observe that both $K$ and $\tilde K$ solve the initial 
value problem \eqref{E:one} on their respective manifolds with matching
initial conditions.
\end{remark}

%%%%%%%%%%%%%%%%%%%%%%%%%%%%%%%%%%%%%%%%%%%%%%%%

\subsection{Long-time behavior}\label{S:long-time}
In this section we will assume that the $L^2$--closure of $\D^+ (\tilde X)$ 
is invertible and derive certain uniform estimates on heat kernels over $\zp$.

Let $K(t;x,y)$ be the smoothing kernel of the operator $e^{-t \D^-\D^+}$ on 
$\zp$ and let $K_0(t;x,y) = K(t;x,y) - K_{P_+}(x,y)$, where $P_+$ is the 
projector onto (the finite dimensional) $\ker \D^+(\zp)$. 

\begin{proposition}\label{P:mu}
There exist positive constants $\mu$ and $C$ such that, for all $x, y \in \zp$
and all $t \ge 1$, 
\[
|K_0 (t;x,y)|\;\le\;C e^{-\mu t}.
\]
\end{proposition}

Before we go on to prove this proposition, we will need a few preliminary results. The operator $\D^-\D^+$ on $\zp$ will be temporarily called $Q$; note that $\ker Q = \ker \D^+$. We will first estimate the operator norm of $e^{-t Q}$ and then use the bounded geometry condition to derive the pointwise estimate as claimed.

\begin{lemma} 
Suppose that $\D^+ (\tilde X)$ is invertible then $Q$ has only discrete 
spectrum near zero (the ``discrete spectrum'' here means ``finitely many 
eigenvalues of finite multiplicity'').
\end{lemma}

\begin{proof}
Since the operator $\D^+ (\tilde X)$ has bounded inverse, one can rely on
the usual parametrix argument; see for instance \cite[Lemma 6.2]{DW}.
\end{proof}

\begin{lemma}
Let $\mu > 0$ be the smallest non-zero eigenvalue of $Q$. For any integer $k 
\ge 0$, there is a constant $C_1 > 0$ such that\; $\| Q^k (e^{-t Q} - P_+) \|
\; \le\; C_1\,e^{-\mu t}$ for all $t > 1$.
\end{lemma}

\begin{proof}
From functional analysis we obtain $\| Q^k (e^{-t Q} - P_+) \|\; \le\; 
\sup\,\{\,\lambda^k e^{-\lambda t}\,|\allowbreak \mu \le \lambda\,\}$.
This supremum equals $\mu^k e^{-\mu t}$ if $t \ge k/\mu$,\; and 
$(k/t)^k e^{-k}$ otherwise. The result now follows. 
\end{proof}

This completes the derivation of the operator norm estimates. To pass to pointwise estimates, we need two more lemmas. 

\begin{lemma}
For all even integers $k > 0$, the operator $(I + Q)^{\,k/2}: L^2_{m+k} 
(\zp;\Sm^+) \to L^2_m (\zp;\Sm^+)$ is an isomorphism.
\end{lemma}

\begin{proof}
Since $\ker (I + Q) = 0$ this is immediate if we define the Sobolev space $L^2_m(\zp;\Sm^+)$ for $m \ge 0$ as the completion of $C^{\infty}_0(\zp;\Sm^+)$ in the norm 
\[
\|s\|_{L^2_m}\;=\;\left(\,\|s\|^2 +\|\D s\|^2 +\ldots +\|\D^m s\|^2\,\right)^{1/2},
\]
and for $m < 0$ as the dual space of $L^2_{-m} (\zp;\Sm^+)$, cf. \cite[Definition 2.6]{roe1}. This definition is equivalent to the standard one because of the bounded geometry condition. 
\end{proof}

\begin{lemma}
For any $k > n/2$ and $x \in \zp$ we have $\delta_x \in L^2_{-k}(\zp;\Sm^+)$. 
Moreover, there is a constant $C_2 > 0$ such that, for all $x \in \zp$, 
\[
\|\,\delta_x\,\|_{L^2_{-k}}\, \le \; C_2.
\]
\end{lemma}

\begin{proof}
The first claim is standard. The second follows using the bounded geometry 
condition as in the proof of Proposition 5.4 of \cite{roe1}.
\end{proof}

\begin{proof}[Proof of Proposition \ref{P:mu}] Let $k > n/2$ be an 
even integer. Since $(I + Q)^{\;k/2}: L^2 (\zp;\Sm^+) \to L^2_{-k} (\zp;\Sm^+)$ 
is an isomorphism, we conclude that
\[
\| (I + Q)^{-k/2} \delta_x \|_{L^2}\; \le\; C_3\, \|\,\delta_x\,\|_{L^2_{-k}}
\]
and therefore
\[
\| (I + Q)^{-k/2} \delta_x \|_{L^2}\; \le\, C_4.
\]
Next, write
\begin{multline}\notag
K_0 (t;x,y)\;=\; (\delta_x,(e^{-tQ} - P_+) \delta_y)\;= \\ 
(\delta_x,(I + Q)^{-k/2} (I + Q)^{\,k} (e^{-tQ} - P_+) (I + Q)^{-k/2} \delta_y)
\end{multline}
using the fact that $Q$ and $e^{-tQ} - P_+$ commute with each other. Since 
$I + Q$ is self-adjoint, we can write
\[
K_0 (t;x,y)\; =\; ((I + Q)^{-k/2}\delta_x,(I + Q)^{\,k} (e^{-tQ} - P_+) 
(I + Q)^{-k/2}\delta_y)
\]
and estimate 
\[
|K_0 (t;x,y)|\;\le\;C_4^2\;\|\,(I+Q)^{\,k} (e^{-tQ} - P_+)\|\;\le\;C\,e^{-\mu t}.
\]
\end{proof}

\begin{proposition}\label{P:Ktilde}
There exist positive constants $\mu$ and $C$ such that, for all $x, y \in 
\tilde X$ and all $t \ge 1$,
\[
|\tilde K(t;x,y)|\;\le\;C e^{-\mu t}.
\]
\end{proposition}

\begin{proof}
Identical to the proof of Proposition \ref{P:mu}, keeping in mind that 
$\tilde X$ has bounded geometry and that $\ker \D^+(\tilde X) = 0$.
\end{proof}

\begin{proposition}\label{P:KP+}
There exist positive constants $C$ and $\delta$ such that for all $x \in W_0$
and $k \ge 1$, one has
\[
|K_{P_+} (x,x)|\;\le\; C\,e^{-\delta\,k}.
\]
\end{proposition}

\begin{proof}
Let $\phi_i$ be an orthonormal basis in $\ker \D^+(\zp)$ then $|K_{P_+}(x,x)| 
= \sum\,|\phi_i (x)|^2$. Let $h: \zp\to \mathbb R$ be a smooth function whose 
restriction to $\txp$ has the property that $h(x + 1) = h(x) + 1$ with 
respect to the covering transformation action. According to \cite{MRS}, the 
invertibility of $\D^+(\tilde X)$ implies that there is a small $\delta > 0$ 
such that the spinors $e^{\,\delta\, h(x)}\phi_i(x)$ form a basis in $\ker 
(\D^+(\zp) - \delta\,dh)$. The kernel of the projector onto $\ker (\D^+(\zp) 
- \delta\,df)$ is uniformly bounded by \cite[Proposition 2.9]{roe1} hence so
are the spinors $e^{\,\delta\, h(x)}\phi_i(x)$. Therefore, one can find $C > 0$ 
such that 

\[
|K_{P_+}(x,x)| = \sum\, |\phi_i (x)|^2 = e^{-2\delta\,h(x)}\,
\sum\,|e^{\delta\,h(x)}\phi_i(x)|^2 \le C\,e^{-2\delta\,h(x)},
\]

\bigskip\noindent
and the result follows. 
\end{proof}

\bigskip

%%%%%%%%%%%%%%%%%%%%%%%%%%%%%%%%%%%%%%%%%%%%%%%%%

\end{document}